\documentclass[11pt,twoside]{article}
\usepackage{times}
\usepackage{amsmath,amssymb}
\usepackage{color}
\usepackage{amsthm}
\usepackage{tikz}

\allowdisplaybreaks
 \def \no{\nonumber}

\def\R {\Bbb R}

\def\p{\partial}
\def\ve{\varepsilon}
\def\f{\frac}

\def\al{\alpha}
\def\t{\tilde}

\def\g{\gamma}

\def\dl{\delta}

\def\ds{\displaystyle}

\newcommand{\md}{\mathrm{d}}
\newcommand{\lc}{\left(}
\newcommand{\rc}{\right)}
 \allowdisplaybreaks

\begin{document}
 \footskip=0pt
 \footnotesep=2pt
\let\oldsection\section
\renewcommand\section{\setcounter{equation}{0}\oldsection}
\renewcommand\thesection{\arabic{section}}
\renewcommand\theequation{\thesection.\arabic{equation}}
\newtheorem{claim}{\noindent Claim}[section]
\newtheorem{theorem}{\noindent Theorem}[section]
\newtheorem{lemma}{\noindent Lemma}[section]
\newtheorem{proposition}{\noindent Proposition}[section]
\newtheorem{definition}{\noindent Definition}[section]
\newtheorem{remark}{\noindent Remark}[section]
\newtheorem{corollary}{\noindent Corollary}[section]
\newtheorem{example}{\noindent Example}[section]

\title{On the global solution problem of semilinear generalized Tricomi equations, II}

\author{He,
Daoyin$^{1, 2*}$;\qquad Witt, Ingo$^{2}$; \qquad Yin,
Huicheng$^{3, }$\footnote{Daoyin He (daoyin.He@mathematik.uni-goettingen.de) and  Huicheng Yin
(huicheng@nju.edu.cn, 05407@njnu. edu.cn) are
supported by the NSFC (No.~11571177),
and by the Priority
Academic Program Development of Jiangsu Higher Education
Institutions.  Ingo Witt (iwitt@uni-math.gwdg.de) was partly supported by the DFG via the Sino-German project ``Analysis of PDEs and application".
}\vspace{0.5cm}\\
\small 1.
Department of Mathematics and
IMS, Nanjing University, Nanjing 210093, China.\\
\small 2.
Mathematical Institute, University of G\"{o}ttingen,
Bunsenstr.~3-5, D-37073 G\"{o}ttingen, Germany.\\
\small 3.
School of Mathematical Sciences, Nanjing Normal University, Nanjing 210023, China.\\}
\vspace{0.5cm}

\date{}
\maketitle
% \vskip 0.2in
\centerline{}
\vskip 0.3 true cm

\centerline {\bf Abstract} \vskip 0.3 true cm

This paper is a continuation of our recent work in \cite{HWYin}. In \cite{HWYin},
for the semilinear generalized
Tricomi equation $\partial_t^2 u-t^m \Delta u=|u|^p$ with initial data $\big(u(0,x), \p_t u(0,x)\big)=\big(u_0(x), u_1(x)\big)$,
$t\ge 0$, $x\in\Bbb R^n$ $(n\ge 3)$, $p>1$ and  $m\in\Bbb N$, we have shown that
there exists a critical exponent $p_{crit}(m, n)>1$ such that the solution $u$ generally blows up in
finite time when $1 < p < p_{crit}(m, n)$; and meanwhile there exists a conformal exponent
$p_{conf}(m, n)$ $\big(> p_{crit}(m, n)\big)$ such that the solution $u$ exists globally when $p\ge p_{conf}(m, n)$ provided
that $(u_0(x), u_1(x))$ are small.
In the present paper, we shall prove that the small data solution $u$ of $\partial_t^2 u-t^m \Delta u=|u|^p$
exists globally when $p_{crit}(m,n)<p<p_{conf}(m,n)$. Therefore, collecting the results in this paper and \cite{HWYin},
we have given a basically systematic study on the blowup or global existence of small data solution $u$ to
the equation $\partial_t^2 u-t^m \Delta u=|u|^p$  for $n\ge 3$. 

\vskip 0.2 true cm

 \vskip 0.2 true cm


\section{Introduction}
In this paper, we continue to study the global Cauchy problem for the following
semilinear generalized Tricomi equation:
\begin{equation}
\left\{ \enspace
\begin{aligned}
&\partial_t^2 u-t^m \Delta u =|u|^p,  \\
&u(0,x)=\ve u_0(x), \quad \partial_{t} u(0,x)=\ve u_1(x), \\
\end{aligned}
\right.
\label{equ:original}
\end{equation}
where $t\ge 0$,
$x=(x_1, ..., x_n)\in\Bbb R^n$ $(n\ge 3)$,  $m\in\Bbb N$, $p>1$, $\ve>0$ is sufficiently small,
and $u_i(x)\in C_0^{\infty}\big(B(0,M-1)\big)$
$(i=0,1)$ with $B(0, M-1)=\{x: |x|=\sqrt{x_1^2+...+x_n^2}<M-1\}$ and $M>1$.
For the local existence and regularity of solution $u$
to problem \eqref{equ:original} under weaker regularity assumptions on $(u_0, u_1)$, the reader may
consult \cite{Rua1}-\cite{Rua4}, \cite{Yag2} and the
references therein; here we shall not discuss this problem.
Since the local existence of  solution $u$
to \eqref{equ:original} with minimal regularities has been established in \cite{Rua4}, without loss of generality,
as in \cite{HWYin} we only focus on the global solution problem of \eqref{equ:original} starting from some
positive time $T_0>0$. Therefore, it is plausible that one utilizes the nonlinear function
$F_p(t,u)=\big(1-\chi(t)\big)F_p(u)+\chi(t)|u|^p$ instead of $|u|^p$ in \eqref{equ:original}, where
$F_p(u)$ is a $C^{\infty}-$smooth function with $F_p(0)=0$ and
$|F_p(u)|\leq C(1+|u|)^{p-1}|u|$, and $\chi(s)\in C^{\infty}(\Bbb R)$
with \(\chi(s)=
\left\{ \enspace
\begin{aligned}
1, \quad &s\geq T_0, \\
0, \quad &s\leq T_0/2.
\end{aligned}
\right.\)
\quad  Correspondingly, we shall study the following problem instead
of \eqref{equ:original}
\begin{equation}
\left\{ \enspace
\begin{aligned}
&\partial_t^2 u-t^m \Delta u =F_p(t,u),  \\
&u(0,x)=\ve u_0(x), \quad \partial_{t} u(0,x)=\ve u_1(x).\\
\end{aligned}
\right.
\label{equ:1.2}
\end{equation}

In \cite{HWYin},  we have determined a critical exponent $p_{crit}(m,n)$ and a conformal
exponent $p_{conf}(m,n)$ $(>p_{crit}(m,n)\big)$ for \eqref{equ:original} or \eqref{equ:1.2} as follows:
$p_{crit}(m,n)$ is the positive root of the algebraic equation
\[\Big((m+2)\frac{n}{2}-1\Big)p^2+\Big((m+2)(1-\frac{n}{2})-3\Big)p-(m+2)=0,\]
and $p_{conf}(m,n)=\frac{(m+2)n+6}{(m+2)n-2}$.
Subsequently it is shown that
the solution $u$ of \eqref{equ:original} or \eqref{equ:1.2} generally blows up in
finite time when $1 < p < p_{crit}(m, n)$, and
meanwhile $u$ exists globally when $p\ge p_{conf}(m, n)$ for small $\ve>0$.
In the present paper, we shall prove that the small data solution $u$ of (1.2)
exists globally when $p_{crit}(m,n)<p< p_{conf}(m,n)$. Therefore, collecting these results,
we have given a basically systematic study on the blowup or global existence of small data solution $u$ to
problem \eqref{equ:1.2} for $n\ge 3$. The main result
in the paper is:

\begin{theorem}\label{thm:global existence}
For $p\in \big(p_{crit}(m,n), p_{conf}(m,n)\big)$ and small
$\ve>0$, problem \eqref{equ:1.2} has a  global solution $u$ such that
\begin{equation}\label{equ:1.3}
\left(1+\big|\phi(t)^2-|x|^2\big|\right)^{\gamma}u\in L^{p+1}(\mathbb{R}^{1+n}_+),
\end{equation}
where  $\phi(t)=\f{2}{m+2}t^{\f{m+2}{2}}$, and the positive constant $\gamma$ fulfills
\begin{equation}\label{equ:1.4}
\f{1}{p(p+1)}<\gamma<\f{\big((m+2)n-2\big)p-\big((m+2)n+2\big)}{2(m+2)(p+1)}+\f{m}{(m+2)(p+1)}.
\end{equation}
\end{theorem}

\begin{remark}
Note that for $p>p_{crit}(m,n)$,  one easily has
\[\f{1}{p(p+1)}<\f{\big((m+2)n-2\big)p-\big((m+2)n+2\big)}{2(m+2)(p+1)}+\f{m}{(m+2)(p+1)}.\]
This implies that condition \eqref{equ:1.4} makes sense.
\end{remark}

\begin{remark}
If we take $m=0$ in \eqref{equ:original}, then problem \eqref{equ:original} becomes
\begin{equation}
\left\{ \enspace
\begin{aligned}
&\partial_t^2 u-\Delta u =|u|^p,  \\
&u(0,x)=\ve u_0(x), \quad \partial_{t} u(0,x)=\ve u_1(x).\\
\end{aligned}
\right.
\label{equ:1.5}
\end{equation}
Let
$p_1(n)$ denote the positive root of the quadratic equation
\begin{equation}\label{1.6}
\left(n-1\right)p^2-\left(n+1\right)p-2=0.
\end{equation}
W.Strauss \cite{Strauss} made the following {\bf conjecture:}
If $p>p_1(n)$, then the small data solution $u$ of \eqref{equ:1.5}
exists globally. If $1<p<p_1(n)$, then the solution $u$ of
\eqref{equ:1.5} generally blows up in finite time.

So far, Strauss' conjecture has basically been solved. For examples,
when $1<p\leq p_1(n)$ and the initial data $\left(u_0,
u_1\right)$ are non-negative, blowup for the solution $u$ of \eqref{equ:1.5} has been
shown, while, for $p > p_1(n)$ and some other restrictions on the power $p$, global existence of small data solution
$u$ of \eqref{equ:1.5} has also been established (see
\cite{Gla1,Gla2,Gls, Joh,Gl2,Sid,Strauss,Yor,Zhou} and the references
therein).
\end{remark}

\begin{remark}
For the cases of $n=1,2$ or $p=p_{crit}(m,n)$ with $n\ge 3$ in problem \eqref{equ:original},
we shall study the problem of blowup or global existence of small data solution $u$ in
our forthcoming paper \cite{HWYin0}.
\end{remark}

\begin{remark}
If $m=0$ is chosen in problem \eqref{equ:1.5}, then the estimates in Theorem~\ref{thm:global existence}
are coincident with the ones in Theorem 1.1 of \cite{Gls}.
\end{remark}

\begin{remark}
As pointed out in \cite{Rei1}, \cite{HWYin} and so on, for large $t>0$,
\eqref{equ:original} is actually equivalent to the following problem
\begin{equation}\label{equ:W}
\left\{ \enspace
\begin{aligned}
&\p_t^2 u-\Delta u+\f{\mu_m}{1+t}\p_t u=(1+t)^{-\al_m}|u|^p,\\
&u(0,x)=\ve u_0(x), \quad \partial_{t} u(0,x)=\ve u_1(x), \\
\end{aligned}
\right.
\end{equation}
where $\mu_m=m/(m+2)$ and $\al_m=2m/(m+2)$. With respect to the problem
\begin{equation}\label{equ:eff}
\left\{ \enspace
\begin{aligned}
&\partial_t^2 u-\Delta u +\f{\mu}{1+t}\,\p_tu=|u|^p, \\
&u(0,x)=u_0(x), \quad \partial_{t} u(0,x)=u_1(x),
\end{aligned}
\right.
\end{equation}
where $\mu>0$, $p>1$, $n\ge 1$, and $u_i\in
C_0^{\infty}(\R^n)$ ($i=0, 1$), so far it has still been an interesting
open problem how to determine explicitly a critical value $p_{c}(\mu,n)$  for $\mu\approx1$
so that problem \eqref{equ:eff} has a global small data solution $u$ for
$p>p_c(\mu, n)$, while solution $u$ of \eqref{equ:eff} generally blows up in
finite time when $1<p\le p_c(\mu,n)$ (see \cite{Rei2, Rei1, Zhai, Nish, Wir-1, Wir-2}).
Motivated by the techniques in \cite{HWYin} and this paper, we shall systematically study
problem \eqref{equ:eff} for $1/3\le \mu<1$ and $p>1$ in the future.
\end{remark}

\begin{remark}
For the semilinear generalized Tricomi equation $\p_t^2u-t^m\Delta u=f(t,x,u)$
with $m\in\Bbb N$ and $x=(x_1, \cdot\cdot\cdot, x_n)$, and for some certain assumptions on the function
$f(t,x,u)$, the authors in \cite{Gva} and \cite{Lup1}-\cite{Lup4}
have obtained some remarkable results on the existence and uniqueness of solution $u$
in {\bf bounded domains} under Tricomi, Goursat or Dirichlet boundary conditions respectively
in the mixed type case, in the degenerate hyperbolic setting or in the degenerate
elliptic setting.
\end{remark}

\begin{remark}
Consider a related semilinear problem to \eqref{equ:original} as follows
\begin{equation}
\left\{ \enspace
\begin{aligned}
&\partial_t^2 u-t|t|^{\mu-1} \Delta u +u|u|^{p-1}=0,  \\
&u(0,x)=\ve u_0(x), \quad \partial_{t} u(0,x)=\ve u_1(x), \\
\end{aligned}
\right.
\label{equ:1.9}
\end{equation}
where $t\ge 0$,
$x=(x_1, ..., x_n)\in\Bbb R^n$ $(n\ge 3)$,  $\mu>0$, $p>1$, and $u_i(x)\in C_0^{\infty}(\Bbb R^n)$.
The authors in \cite{Lup3} have dealt with the energy estimates of classical solution $u$ when the exponent p
is critical or supercritical (i.e., $p\ge2n^*/(n^*-2)$, where $n^*=n(\mu+2)/2$ is the homogeneous dimension
of the operator $\partial_t^2-t|t|^{\mu-1} \Delta$).
However, by our knowledge, so far there are no systematic results on the global existence or uniqueness
of solution $u$ to problem ~\eqref{equ:1.9}.
\end{remark}

Let's recall some known results on  problem \eqref{equ:original} or \eqref{equ:1.2}. Under the restricted conditions
\begin{equation}\label{equ:ugly}
\left\{ \enspace
\begin{aligned}
&\f{(n+1)(p-1)}{p+1}\le\f{m}{m+2},\\
&\left(\f{2}{p-1}-\f{n(m+2)}{2(p+1)}\right)p\le 1,\\
&\f{2(p+1)}{p(p-1)n(m+2)}\le\f{1}{p+1}\le\f{m+4}{(n+1)(p-1)(m+2)}
\end{aligned}
\right.
\end{equation}
(corresponding to (1.8) and (1.12) of \cite{Yag2} with $k=m/2$, $\al=p-1$ and
$\beta=\f{2}{p-1}-\f{n(k+1)}{p+1}$) it was shown in \cite[Theorem~1.2]{Yag2}
that problem (1.1) has a global small data solution
$u\in C([0, \infty), L^{p+1}(\R^n)) \cap C^1([0, \infty), {\mathcal D}'(\R^n))$.
On the other hand, under the conditions $\int_{\R^n} u_1(x)\md x \\
>0$ and
\begin{equation}\label{equ:blow}
1<p<\f{(m+2)n+2}{(m+2)n-2},
\end{equation}
it was shown in \cite[Theorem~1.3]{Yag2} that (1.1)
has no global solution $u\in C([0, \infty),
L^{p+1}(\R^n))$. Here we
point out that \eqref{equ:blow} comes from condition (1.15) of
\cite{Yag2}. As emphasized in \cite{HWYin}, with respect to
the blowup or global existence of solution $u$ to problem (1.1),
there are some gaps for the scope of power $p$ in \cite{Yag2}. In this and the previous paper \cite{HWYin},
motivated by the Strauss conjecture and \cite{Gls}, \cite{Gl2}, we shall systematically
study  problem (1.1) or (1.2) for $n\ge 3$.

There are extensive results concerning the Cauchy problem
for both linear and semilinear generalized Tricomi equations. For
instance, for the linear generalized Tricomi equation,
the authors in \cite{Gelf}, \cite{Yag1} and \cite{Yag3} have computed its
fundamental solution explicitly. More recently, the authors in
\cite{Rua1, Rua2, Rua3, Rua4} established the local existence as well
as the singularity structure of low regularity solutions to the
semilinear equation $\partial_t^2u -t^m\triangle u=f(t,x,u)$ in the
degenerate hyperbolic region and the elliptic-hyperbolic mixed region,
respectively, where $f$ is a $C^1$ function and has compact support
with respect to the variable $x$. By establishing some classes of $L^p$-$L^q$  estimates
for the solution $v$ of linear equation $\partial_t^2v -t^m\triangle v=F(t,x)$,
the author in \cite{Yag2} obtained a
series of interesting results about the global existence or the
blowup of solutions to problem (1.1) when the exponent
$p$ belongs to a certain range, however, there was a
gap between the global existence interval and the blowup interval;
moreover, the critical exponent $p_{crit}(m,n)$ was not
determined there.

We now comment on the proof of Theorem 1.1.
To prove Theorem 1.1, motivated by \cite{Gls, Gl2}, where some basic
weighted Strichartz estimates with the weight $1+|t^2-|x|^2|$ were obtained for the linear wave operator
$\p_t^2-\triangle$, we
require to establish some Strichartz estimates with the characteristic weight
$(\phi(t)+M)^2-|x|^2$  (here and below $\phi(t)\equiv 2/(m+2)t^{\f{m+2}{2}}$) for the generalized
Tricomi operator $\p_t^2-t^m\Delta$. Here we point out that in order to derive the
weighted Strichartz estimates for $\p_t^2-\triangle$, the authors in \cite{Gls} utilize
the crucial conformal transformation method to change the characteristic cone $\{(t,x): t>0, |x|<t\}$
into the half cone $\{(t,x): t>0, |x|<\f{t}{2}\}$ meanwhile the wave operator
$\p_t^2-\triangle$ keeps invariant (in this case, the related weight $1+|t^2-|x|^2|$
is equivalent to $1+t^2$ for large $t>0$). However, it seems that this conformal transformation method
used in  \cite{Gls} is not available for us since it is difficult to find a suitable  conformal transformation
to let the characteristic cusp cone $\{(t,x): t>0, |x|<\phi(t)\}$
become the half cusp cone $\{(t,x): t>0, |x|<\f{\phi(t)}{2}\}$ and simultaneously keep the generalized Tricomi
operator $\p_t^2-t^m\Delta$ invariant. Thanks to  some involved analysis and careful observations,
we are able to overcome the technical difficulties related to
degeneracy  and eventually establish the expected weighted
Strichartz estimates for problem \eqref{equ:1.2}. In addition, some elementary but important estimates
used in $\S 4$ are put in the appendix.

In Section 2, the following linear problem is studied
\begin{equation*}
\begin{cases}
&\partial_t^2 v-t^m\triangle v=0, \\
&v(0,x)=f(x),\quad \partial_tv(0,x)=g(x), \\
\end{cases}
\label{equ:0.1}
\end{equation*}
where $f, g\in C_0^\infty(\mathbb{R}^n)$,  $\operatorname{supp}\ (f,g)\subseteq \{x: |x|\leq M-1\}$ for some fixed constant
$M>1$. By introducing certain Fourier
integral operators associated to $\partial_t^2 v-t^m\triangle$ and taking delicate analysis,
we establish Strichartz-type estimates of the form
\begin{equation}\label{equ:0.2}
\left\|\big((\phi(t)+M)^2-|x|^2\big)^\gamma v\right\|_{L^q(\mathbb{R}^{1+n}_+)}\leq C(\parallel f\parallel_{W^{\f{n}{2}+\f{1}{m+2}+\delta,1}(\mathbb{R}^n)}+\parallel g\parallel_{W^{\f{n}{2}-\f{1}{m+2}+\delta,1}(\mathbb{R}^n)}),
\end{equation}
where $q>\f{2((m+2)n-m)}{(m+2)n-2}$, $0<\gamma<\f{(m+2)n-2}{2(m+2)}-\f{(m+2)n-m}{(m+2)q}$,
and $0<\delta<\f{n}{2}+\f{1}{m+2}-\gamma-\f{1}{q}$.

In Section 3- Section 5, we focus on the study of the problem
\begin{equation*}
\begin{cases}
&\partial_t^2 w-t^m\triangle w=F(t,x), \\
&w(0,x)=0,\quad \partial_tw(0,x)=0,
\end{cases}
\label{equ:0.3}
\end{equation*}
where $F(t,x)\equiv0$ when
$|x|>\phi(t)+M-1$ and $F(t,x)\in C^{\infty}([0, T_0]\times\Bbb R^n)$ for some fixed
number $T_0>0$ $(T_0<1)$. We eventually establish
\begin{equation}
\begin{split}
\Big\|\Big(\big(\phi(t)+M\big)^2-|x|^2\Big)^{\gamma_1}w\Big\|_{L^q([\f{T_0}{2}, \infty)\times \mathbb{R}^{n})}\leq C
\Big\|\Big(\big(\phi(t)+M\big)^2-|x|^2\Big)^{\gamma_2}F\Big\|_{L^{\frac{q}{q-1}}([\f{T_0}{2}, \infty)\times \mathbb{R}^{n})},
\end{split}
\label{equ:0.4}
\end{equation}
where the constants $\gamma_1$ and $\gamma_2$ satisfy $0<\gamma_1<\frac{(m+2)n-2}{2(m+2)}-\frac{(m+2)n-m}{(m+2)q}$ and
$\gamma_2>\frac{1}{q}$, and $q>\f{2((m+2)n-m)}{(m+2)n-2}$.
To derive \eqref{equ:0.4}, we shall split the integral domain $\{(t,x): \phi(t)^2-|x|^2\leq1\}$
in the corresponding Fourier integral operators into some pieces,
which correspond to the ``relatively large time" part and the ``relatively small time"
part respectively. By introducing and analyzing such kinds of Fourier integral operators for $z\in\Bbb C$,
\newpage
\begin{align*}
%\begin{split}
(T_zg)(t,x)=&(z-\frac{(m+2)n+2}{2(m+2)})e^{z^2}\int_{\mathbb{R}^n}\int_{\mathbb{R}^n}e^{i[(x-y)\cdot\xi-(\phi(t)-|y|)|\xi|]} \\
&\qquad\qquad\times\big(1+\phi(t)|\xi|\big)^{-\frac{m}{2(m+2)}}g(y)\frac{\md\xi}{|\xi|^z}\md y,\\
(\tilde{T}_zg)(t,x)=&\int_{\Bbb R^n}\int_{\Bbb R^n}e^{i[(x-y)\cdot\xi-(\phi(t)-|y|)|\xi|]}
\big(1+\phi(t)|\xi|\big)^{-\frac{m}{2(m+2)}}\\
&\qquad\qquad  \times \big(1+\phi(s)|\xi|\big)^{-\frac{m}{2(m+2)}}g(y)\frac{\md\xi}{|\xi|^{z}}\md y,\\
%\end{split}
\label{equ:4.58}
\end{align*}
and combining with the complex interpolation methods, we
ultimately obtain \eqref{equ:0.4}. Based on \eqref{equ:0.2}
and \eqref{equ:0.4}, by the standard contraction mapping principle, we
complete the proof of Theorem 1.1 in Section 6.

%This paper is organized as follows: In $\S 2$,
%a weighted Strichartz inequality for the linear homogeneous
%equation $\partial_t^2 v-t^m\triangle v=0$ is established. In $\S 3$, by assuming two endpoint Strichartz inequalities,  we derive
%a general weighted Strichartz inequality for the linear inhomogeneous
%equation $\partial_t^2 v-t^m\triangle v=F$. In $\S 4$ and  $\S 5$, we will complete the proofs
%of the first and the second endpoint inequality listed in $\S 3$, respectively.
%Theorem 1.1 is proved in $\S 6$. In addition, some elementary but important estimates
%used in $\S 4$ are put in the appendix.

\section{The weighted Strichartz inequality for homogeneous generalized Tricomi equation}

In this section, our main purpose is to establish a weighted Strichartz inequality for the homogeneous
generalized Tricomi equation:
\begin{equation}
\begin{cases}
&\partial_t^2 v-t^m\triangle v=0, \\
&v(0,x)=f(x),\quad \partial_tv(0,x)=g(x), \\
\end{cases}
\label{equ:2.1}
\end{equation}
where $f, g\in C_0^\infty(\mathbb{R}^n)$,  $\operatorname{supp}(f,g)\subseteq \{x: |x|\leq M-1\}$ for some fixed constant
$M>1$. We now have the following weighted space-time estimate of Strichartz-type.

\begin{theorem}\label{thm:homogeneous estimate}
For the solution $v$ of (2.1), one then has
\begin{equation}\label{equ:2.2}
\left\|\Big(\big(\phi(t)+M\big)^2-|x|^2\Big)^\gamma v\right\|_{L^q(\mathbb{R}^{1+n}_+)}\leq C(\parallel f\parallel_{W^{\f{n}{2}+\f{1}{m+2}+\delta,1}(\mathbb{R}^n)}+\parallel g\parallel_{W^{\f{n}{2}-\f{1}{m+2}+\delta,1}(\mathbb{R}^n)}),
\end{equation}
where $\phi(t)=\frac{2}{m+2}t^{\frac{m+2}{2}}$,
$q>\f{2((m+2)n-m)}{(m+2)n-2}$, $0<\gamma<\f{(m+2)n-2}{2(m+2)}-\f{(m+2)n-m}{(m+2)q}$,
$0<\delta<\f{n}{2}+\f{1}{m+2}-\gamma-\f{1}{q}$,
and $C$ is a positive constant depending only on $m$, $n$, $q$, $\gamma$ and $\delta$.
\end{theorem}

\begin{proof}
It follows from \cite{Yag2} that the solution $v$ of \eqref{equ:2.1} can be expressed as
\[v(t,x)=V_1(t, D_x)f(x)+V_2(t, D_x)g(x),\]
where the symbols $V_j(t, \xi)$ ($j=1,2$) of the Fourier integral operators $V_j(t, D_x)$  are
\begin{equation}\label{equ:2.3}
\begin{split}
V_1(t,|\xi|)=&\frac{\Gamma(\frac{m}{m+2})}{\Gamma(\frac{m}{2(m+2)})}\biggl[e^{\frac{z}{2}}H_+\Big(\frac{m}{2(m+2)},\frac{m}{m+2};z\Big) +e^{-\frac{z}{2}}H_-\Big(\frac{m}{2(m+2)},\frac{m}{m+2};z\Big)\biggr]
\end{split}
\end{equation}
and
\begin{equation}
\begin{split}
V_2(t,|\xi|)=&\frac{\Gamma(\frac{m+4}{m+2})}{\Gamma(\frac{m+4}{2(m+2)})}t\biggl[
e^{\frac{z}{2}}H_+\Big(\frac{m+4}{2(m+2)},\frac{m+4}{m+2};z\Big)
+e^{-\frac{z}{2}}H_-\Big(\frac{m+4}{2(m+2)},\frac{m+4}{m+2};z\Big)\biggr],
\end{split}
\label{equ:2.4}
\end{equation}
here $z=2i\phi(t)|\xi|$, $i=\sqrt{-1}$, and $H_{\pm}$ are smooth functions of the variable $z$.
By \cite{Tani}, one knows that for $\beta\in\mathbb{N}_0^n$,
\begin{align}
\big| \partial_\xi^\beta H_{+}(\alpha,\gamma;z)
\big|&\leq C(\phi(t)|\xi|)^{\alpha-\gamma}(1+|\xi|^2)^{-\frac{|\beta|}{2}}
\quad if \quad \phi(t)|\xi|\geq 1, \label{equ:2.5} \\
\big| \partial_\xi^\beta H_{-}(\alpha,\gamma;z)\big|&\leq C(\phi(t)|\xi|)^{-\alpha}(1+|\xi|^2)^{-\frac{|\beta|}{2}}
\quad if \quad \phi(t)|\xi|\geq 1. \label{equ:2.6}
\end{align}
To estimate $v$, it  only suffices to deal with $V_1(t, D_x)f(x)$  since the treatment
on $V_2(t, D_x)g(x)$ is similar. Indeed, if one just notices a simple fact of $t\phi(t)^{-\frac{m+4}{2(m+2)}}=
C_m\phi(t)^{-\frac{m}{2(m+2)}}$, it then follows from the expressions
of $V_1(t,\xi)$ and $V_2(t,\xi)$ that the orders of $t$ in $V_1(t,\xi)$ and $V_2(t,\xi)$ are the same.
Therefore, without loss of generality, we assume $g(x)\equiv 0$ in  \eqref{equ:2.1}
from now on.

Choose a cut-off function $\chi(s)\in C^{\infty}(\Bbb R)$
with $\chi(s)=
\left\{ \enspace
\begin{aligned}
1, \quad &s\geq2 \\
0, \quad &s\leq1
\end{aligned}
\right.$. Then
\begin{equation}
\begin{split}
V_1(t,|\xi|)\hat{f}(\xi)&=\chi(\phi(t)|\xi|)V_1(t,|\xi|)\hat{f}(\xi)+(1-\chi(\phi(t)|\xi|))V_1(t,|\xi|)\hat{f}(\xi) \\
&=:\hat{v}_1(t,\xi)+\hat{v}_2(t,\xi).
\end{split}
\label{equ:2.7}
\end{equation}
By \eqref{equ:2.3}, \eqref{equ:2.5} and \eqref{equ:2.6}, we derive that
\begin{equation}
{v}_1(t,x)=C_m\biggl(\int_{\mathbb{R}^n}e^{i(x\cdot\xi+\phi(t)|\xi|)}a_{11}(t,\xi)\hat{f}(\xi)\md\xi+
\int_{\mathbb{R}^n}e^{i(x\cdot\xi-\phi(t)|\xi|)}a_{12}(t,\xi)\hat{f}(\xi)\md\xi\biggr), \label{equ:2.8}
\end{equation}
where $C_m>0$ is a generic constant depending on $m$, and for $\beta\in\mathbb{N}_0^n$,
\begin{equation*}
\big| \partial_\xi^\beta a_{1l}(t,\xi)\big|\leq C_{l\beta}|\xi|^{-|\beta|}\big(1+\phi(t)|\xi|\big)^{-\frac{m}{2(m+2)}},
\qquad l=1,2.
\end{equation*}
Next we analyze $v_2(t,x)$. It follows from \cite{Erd1} or \cite{Yag2} that
\begin{equation*}
V_1(t,|\xi|)=e^{-\frac{z}{2}}\Phi\Big(\frac{m}{2(m+2)},\frac{m}{m+2};z\Big), %\label{equ:2.9}
\end{equation*}
where $\Phi$ is the confluent hypergeometric function which is analytic with respect to the variable
$z=2i\phi(t)|\xi|$. Then
\begin{equation*}
\Big|\partial_\xi\big\{\big(1-\chi(\phi(t)|\xi|)\big)V_1(t,|\xi|)\big\}\Big|\leq C(1+\phi(t)|\xi|)^{-\frac{m}{2(m+2)}}|\xi|^{-1}.
\end{equation*}
Similarly, one has
\begin{equation*}
\Big|\partial_\xi^{\beta}\big\{\big(1-\chi(\phi(t)|\xi|)\big)V_1(t,|\xi|)\big\}\Big|\leq
C(1+\phi(t)|\xi|)^{-\frac{m}{2(m+2)}}|\xi|^{-|\beta|}.
\end{equation*}
Thus we arrive at
\begin{equation}
v_2(t,x)=C_m\biggl(\int_{\mathbb{R}^n}e^{i(x\cdot\xi+\phi(t)|\xi|)}a_{21}(t,\xi)\hat{f}(\xi)\md\xi
+\int_{\mathbb{R}^n}e^{i(x\cdot\xi-\phi(t)|\xi|)}a_{22}(t,\xi)\hat{f}(\xi)\md\xi\biggr), \label{equ:2.10}
\end{equation}
where, for $\beta\in\mathbb{N}_0^n$,
\begin{equation*}
\big| \partial_\xi^\beta a_{2l}(t,\xi)\big|\leq C_{l\beta}\big(1+\phi(t)|\xi|\big)^{-\frac{m}{2(m+2)}}|\xi|^{-|\beta|},
\qquad l=1,2.
\end{equation*}
Substituting \eqref{equ:2.8} and \eqref{equ:2.10} into \eqref{equ:2.7} yields
\[V_1(t, D_x)f(x)=C_m\biggl(\int_{\mathbb{R}^n}e^{i(x\cdot\xi+\phi(t)|\xi|)}a_1(t,\xi)\hat{f}(\xi)\md\xi
+\int_{\mathbb{R}^n}e^{i(x\cdot\xi-\phi(t)|\xi|)}a_2(t,\xi)\hat{f}(\xi)\md\xi\biggr),\]
where $a_l$ $(l=1,2)$ satisfies
\begin{equation}
|\partial_\xi^\beta a_l(t,\xi)\big|\leq C_{l\beta}\big(1+\phi(t)|\xi|\big)^{-\frac{m}{2(m+2)}}|\xi|^{-|\beta|}. \label{equ:2.11}
\end{equation}
To estimate $V_1(t, D_x)f(x)$, it only suffices to deal with $\int_{\mathbb{R}^n}e^{i(x\cdot\xi+\phi(t)|\xi|)}a_1(t,\xi)\hat{f}(\xi)\md\xi$
since the term $\int_{\mathbb{R}^n}e^{i(x\cdot\xi-\phi(t)|\xi|)}a_2(t,\xi)\hat{f}(\xi)\md\xi$ can be  analogously treated.
Set
\begin{equation*}
(Af)(t,x)=:\int_{\mathbb{R}^n}e^{i(x\cdot\xi+\phi(t)|\xi|)}a_1(t,\xi)\hat{f}(\xi)\md\xi. %\label{equ:2.12}
\end{equation*}
Let $\beta(\tau)\in C_0^\infty(\frac{1}{2},2)$ such that
\begin{equation}
\sum\limits_{j=-\infty}^\infty\beta(\frac{\tau}{2^j}) \equiv1\quad \text{for $\tau\in\mathbb{R}_+$.} \label{equ:2.13}
\end{equation}
To estimate $(Af)(t,x)$, we now study its corresponding dyadic operators
\begin{equation*}
\begin{split}
(A_jf)(t,x)&=\int_{\mathbb{R}^n}e^{i(x\cdot\xi+\phi(t)|\xi|)}\beta(\frac{|\xi|}{2^j})
a_1(t,\xi)\hat{f}(\xi)\md\xi \\
&=:\int_{\mathbb{R}^n}e^{i(x\cdot\xi+\phi(t)|\xi|)}a_j(t,\xi)\hat{f}(\xi)\md\xi, \\
\end{split}
%\label{equ:2.14}
\end{equation*}
where $j\in\Bbb Z$. Note that the kernel of operator $A_j$ is
\begin{equation*}
K_j(t,x;y)=\int_{\mathbb{R}^n}e^{i((x-y)\cdot\xi+\phi(t)|\xi|)}a_j(t,\xi)\md\xi,
\end{equation*}
where $|y|\le M-1$ because of $\operatorname{supp} f\subseteq \{x: |x|\leq M-1\}$. By (3.29) of \cite{Gl2}, we have that for any $N\in\Bbb {\Bbb R^+}$,
\begin{equation}\label{equ:2.15}
\begin{split}
|K_j(t,x;y)|\leq &C_{m,n,N}\lambda_j^{\frac{n+1}{2}}(1+\phi(t)\lambda_j)^{-\frac{m}{2(m+2)}}
\big(\phi(t)+\lambda_j^{-1}\big)^{-\frac{n-1}{2}}\big(1+\lambda_j\big||x-y|-\phi(t)\big|\big)^{-N},
\end{split}
\end{equation}
where $\lambda_j=2^j$. Since the solution $v$ of \eqref{equ:2.1} is smooth and has compact support on the variable
$x$ for any fixed time, one easily knows that (2.2) holds in any fixed domain $[0, T]\times\Bbb R^n$.
Therefore, in order to prove (2.2), it suffices to consider the case of $\phi(t)\gg M$. At this time,
the following two cases will be studied separately.

\subsection{\boldmath$\big||x-y|-\phi(t)\big|\gg M$}\label{sec2:big}

For this case, there exist two positive constants $C_1$ and $C_2$ such that
\[C_1\big||x-y|-\phi(t)\big|\geq\big||x|-\phi(t)\big|\geq C_2\big||x-y|-\phi(t)\big|\gg M.\]
If $j\geq0$, we then take $N=\frac{n}{2}+\frac{1}{m+2}+\delta$ in \eqref{equ:2.15} and obtain
\begin{equation*}
\begin{split}
|K_j(t,x;y)|&\leq C_{m,n,\delta}\lambda_j^{\frac{n+1}{2}-\frac{m}{2(m+2)}}
\phi(t)^{-\frac{n-1}{2}-\frac{m}{2(m+2)}}\lambda_j^{-\frac{n}{2}-\frac{1}{m+2}-\delta}
\big||x|-\phi(t)\big|^{-\frac{n}{2}-\frac{1}{m+2}-\delta} \\
&\leq C_{m,n,\delta}\lambda_j^{-\delta}\big(1+\phi(t)\big)^{-\frac{n-1}{2}
-\frac{m}{2(m+2)}}\big(1+\big||x|-\phi(t)\big|\big)^{-\frac{n}{2}-\frac{1}{m+2}-\delta}.
\end{split}
\end{equation*}
For $j<0$, taking $N=\frac{n}{2}+\frac{1}{m+2}-\delta$ in \eqref{equ:2.15}  we arrive at
\begin{equation*}
\begin{split}
|K_j(t,x;y)|&\leq C_{m,n,\delta}\lambda_j^{\frac{n+1}{2}-\frac{m}{2(m+2)}}
\phi(t)^{-\frac{n-1}{2}-\frac{m}{2(m+2)}}\lambda_j^{-\frac{n}{2}-\frac{1}{m+2}+\delta}
\big||x|-\phi(t)\big|^{-\frac{n}{2}-\frac{1}{m+2}+\delta} \\
&\leq C_{m,n,\delta}\lambda_j^\delta\big(1+\phi(t)\big)^{-\frac{n-1}{2}-\frac{m}{2(m+2)}}\big(1+\big||x|-\phi(t)\big|\big)^{-\frac{n}{2}-\frac{1}{m+2}+\delta}.
\end{split}
\end{equation*}
It follows from $f(x)\in C_0^\infty(\mathbb{R}^n)$ and direct computation that
\begin{equation} \label{equ:2.16}
|A_jf|\leq
\left\{ \enspace
\begin{aligned}
&C_{m,n,\delta}\lambda_j^\delta\big(1+\phi(t)\big)^{-\frac{n-1}{2}-\frac{m}{2(m+2)}}
\big(1+\big||x|-\phi(t)\big|\big)^{-\frac{n}{2}-\frac{1}{m+2}+\delta}\|f\|_{L^1(\Bbb R^n)}, &&j<0,\\
&C_{m,n,\delta}\lambda_j^{-\delta}\big(1+\phi(t)\big)^{-\frac{n-1}{2}-\frac{m}{2(m+2)}}
\big(1+\big||x|-\phi(t)\big|\big)^{-\frac{n}{2}-\frac{1}{m+2}-\delta}\|f\|_{L^1(\Bbb R^n)}, &&j\geq0.
\end{aligned}
\right.
\end{equation}
Summing the right sides of \eqref{equ:2.16}, we get that for large $\phi(t)$ and $\big||x|-\phi(t)\big|$,
\begin{equation}\label{equ:2.17}
|v|\leq C_{m,n,\delta}\big(1+\phi(t)\big)^{-\frac{n-1}{2}-\frac{m}{2(m+2)}}
\big(1+\big||x|-\phi(t)\big|\big)^{-\frac{n}{2}-\frac{1}{m+2}+\delta}\|f\|_{L^1(\Bbb R^n)}.
\end{equation}

\subsection{\boldmath$\big||x-y|-\phi(t)\big|\le C M$}\label{sec2:small}

By the similar method as in Case 1, we can establish that for $t>1$,
\begin{equation}\label{equ:2.18}
\parallel v(t,\cdot)\parallel_{L^\infty(\mathbb{R}^n)}\leq C_{m,n,\delta}\phi(t)^{-\frac{n-1}{2}-\frac{m}{2(m+2)}}\parallel f\parallel_{W^{\frac{n}{2}+\frac{1}{m+2}+\delta,1}(\mathbb{R}^n)},
\end{equation}
where $0<\delta<\frac{n}{2}+\frac{1}{m+2}-\gamma-\frac{1}{q}$ is a positive constant.

Indeed, note that
\begin{equation*}
\begin{split}
|A_jf|&=\bigg|\int_{\mathbb{R}^n}e^{i(x\cdot\xi+\phi(t)|\xi|)}\frac{a_j(t,\xi)}{|\xi|^\alpha}\widehat{|D_x|^\alpha f}(\xi)\md\xi\bigg|,
\end{split}
\end{equation*}
where $\alpha=\frac{n}{2}+\frac{1}{m+2}+\delta$. Then by the stationary phase method, we have
that for $j\ge 0$,
\begin{equation}\label{equ:2.19}
\begin{split}
|A_jf|&\leq C_{m,n,\delta}\lambda_j^{-\alpha}\lambda_j^{\frac{n+1}{2}}
\big(1+\phi(t)\lambda_j\big)^{-\frac{m}{2(m+2)}}\big(\phi(t)+\lambda_j^{-1}\big)^{-\frac{n-1}{2}}
\parallel f\parallel_{W^{\frac{n}{2}+\frac{1}{m+2}+\delta,1}(\mathbb{R}^n)} \\
&\leq C_{m,n,\delta}\lambda_j^{-\delta}\big(1+\phi(t)\big)^{-\frac{n-1}{2}-\frac{m}{2(m+2)}}
\parallel f\parallel_{W^{\frac{n}{2}+\frac{1}{m+2}+\delta,1}(\mathbb{R}^n)}.
\end{split}
\end{equation}
Similarly, for $j<0$, we have
\begin{equation}\label{equ:2.20}
|A_jf|\leq C_{m,n,\delta}\lambda_j^\delta\big(1+\phi(t)\big)^{-\frac{n-1}{2}-\frac{m}{2(m+2)}}
\|f\|_{W^{\frac{n}{2}+\frac{1}{m+2}-\delta,1}(\mathbb{R}^n)}.
\end{equation}
Summing all the terms in \eqref{equ:2.19} and \eqref{equ:2.20} yields \eqref{equ:2.18}.
%\begin{equation}\label{equ:2.21}
%\parallel v\parallel_{L^\infty(\mathbb{R}^n)}
%\end{equation}
%which means that \eqref{equ:2.18} holds.

Therefore, it follows from \eqref{equ:2.17} and \eqref{equ:2.18} that
\begin{equation}\label{equ:2.22}
|v|\leq C_{m,n,\delta}(1+\phi(t))^{-\frac{n-1}{2}-\frac{m}{2(m+2)}}(1+\big||x|-\phi(t)\big|)^{-\frac{n}{2}-\frac{1}{m+2}+\delta}
\|f\|_{W^{\frac{n}{2}+\frac{1}{m+2}-\delta,1}(\mathbb{R}^n)}.
\end{equation}
Next we compute the integral in the left hand side of \eqref{equ:2.2} by using \eqref{equ:2.22} and the polar coordinate transformation.
\begin{equation}\label{equ:2.23}
\begin{split}
&\Big\|\big((\phi(t)+M)^2-|x|^2\big)^\gamma v\Big\|_{L^q(\mathbb{R}^{1+n}_+)}^q \\
&\le C_{m,n,\delta}\|f\|_{W^{\frac{n}{2}+\frac{1}{m+2}-\delta,1}}\int_0^\infty\int_{\mathbb{R}^n}
\bigg(\Big(\big(\phi(t)+M\big)^2-|x|^2\Big)^\gamma\big(1+\phi(t)\big)^{-\frac{n-1}{2}-\frac{m}{2(m+2)}} \\
&\qquad\qquad \qquad\qquad \qquad  \qquad \qquad  \times\Big(1+\big||x|-\phi(t)\big|\Big)^{-\frac{n}{2}-\frac{1}{m+2}+\delta}\bigg)^q dxdt
\\
&\leq C_{m,n,\delta}\|f\|_{W^{\frac{n}{2}+\frac{1}{m+2}-\delta,1}}
\int_0^\infty\int_0^\infty\Big(\big(\phi(t)+M+r\big)^\gamma\big(\phi(t)+M-r\big)^\gamma \\
&\qquad\qquad \qquad  \qquad  \qquad  \times\big(1+\phi(t)\big)^{-\frac{n-1}{2}-\frac{m}{2(m+2)}}
\big(1+|r-\phi(t)|\big)^{-\frac{n}{2}-\frac{1}{m+2}+\delta}\Big)^q r^{n-1}drdt
\\
&\leq C_{m,n,\delta}\|f\|_{W^{\frac{n}{2}+\frac{1}{m+2}-\delta,1}}
\int_0^\infty\int_0^\infty\Big(\big(1+\phi(t)\big)^{-\frac{n-1}{2}-\frac{m}{2(m+2)}+\gamma} \\
&\qquad\qquad  \qquad\qquad \qquad  \qquad  \qquad \quad \times\big(1+|r-\phi(t)|\big)^{\gamma-\frac{n}{2}
-\frac{1}{m+2}+\delta}\Big)^q r^{n-1}drdt.
\end{split}
\end{equation}
Notice that by our assumption, $\gamma-\frac{n-1}{2}-\frac{m}{2(m+2)}<(\frac{m}{m+2}-n)\frac{1}{q}$ holds.
Then we can choose two constants $\sigma>0$ and $\dl>0$ such that
\[\gamma-\frac{n-1}{2}-\frac{m}{2(m+2)}<\Big(\frac{m}{m+2}-n\Big)\frac{1}{q}-\sigma\]
and
\[\Big(\gamma-\frac{n}{2}-\frac{1}{m+2}+\delta\Big)q<-1.\]
Then for some positive constant $\bar{\sigma}>0$, the integral in the last line of \eqref{equ:2.23} can be controlled by
\begin{equation*}
\begin{split}
&\int_0^\infty\int_0^\infty\big(1+\phi(t)\big)^{\frac{m}{m+2}-n-\bar{\sigma}}
\big(1+\big|r-\phi(t)\big|\big)^{-1-\bar{\sigma}}r^{n-1}drdt \\
&\leq C\int_0^\infty\big(1+\phi(t)\big)^{\frac{m}{m+2}-n-\bar{\sigma}}\big(1+\phi(t)\big)^{n-1}dt\\
&\leq C.
\end{split}
\end{equation*}
This, together with \eqref{equ:2.23}, yields (2.2).
\end{proof}

\section{The weighted Strichartz inequality for the inhomogeneous generalized Tricomi equation}

In this section, based on the assumptions of two endpoint inequalities (whose proofs are given in $\S 4$ and $\S 5$
respectively), we shall establish
a general weighted Strichartz inequality for
the following inhomogeneous Tricomi equation
\begin{equation}
\begin{cases}
&\partial_t^2 w-t^m\triangle w=F(t,x), \\
&w(0,x)=0,\quad \partial_tw(0,x)=0.
\end{cases}
\label{equ:3.1}
\end{equation}
Our first result is:

\begin{theorem}\label{thm:inhomogeneous estimate}
For problem \eqref{equ:3.1}, if $F(t,x)\equiv0$ when
$|x|>\phi(t)-1$, then there exist some constants $\gamma_1$ and $\gamma_2$ satisfying $0<\gamma_1
<\frac{(m+2)n-2}{2(m+2)}-\frac{(m+2)n-m}{(m+2)q}$
and $\gamma_2>\frac{1}{q}$, such that
\begin{equation}
\begin{split}
\big\|\big(\phi(t)^2-|x|^2\big)^{\gamma_1}w\big\|_{L^q(\mathbb{R}^{1+n}_+)}\leq C
\big\|\big(\phi(t)^2-|x|^2\big)^{\gamma_2}F\big\|_{L^{\frac{q}{q-1}}(\mathbb{R}^{1+n}_+)},
\end{split}
\label{equ:3.2}
\end{equation}
where $q>\f{2((m+2)n-m)}{(m+2)n-2}$,
and $C>0$ is a positive constant  depending on $m$, $n$, $q$, $\g_1$ and $\gamma_2$.
\end{theorem}

\begin{proof}
To establish \eqref{equ:3.2},  motivated by \cite{Gls}, we shall assume that \eqref{equ:3.2} holds
for two special cases, namely, \eqref{equ:3.2} holds for  the two endpoints $q=q_0=\frac{2((m+2)n+2)}{(m+2)n-2}$
(corresponding to the extreme cases of $\gamma_1=\frac{(m+2)n-2}{2(m+2)}-\frac{(m+2)n-m}{(m+2)q}$, $\gamma_2=\frac{1}{q}$
and $\gamma_1=\gamma_2$ in the inequality \eqref{equ:3.2}) and $q=2$
(corresponding $q=\f{q}{q-1}$ for the inequality \eqref{equ:3.2}):
\begin{equation}
\begin{split}
\big\|\big(\phi(t)^2-|x|^2\big)^{\gamma_1}&w\big\|_{L^{q_0}(\mathbb{R}^{1+n}_+)}\leq C
\big\|\big(\phi(t)^2-|x|^2\big)^{\gamma_2}F\big\|_{L^{\frac{q_0}{q_0-1}}(\mathbb{R}^{1+n}_+)},
\end{split}
\label{equ:3.3}
\end{equation}
where $\gamma_1<\frac{1}{q_0}$ and $\gamma_2>\frac{1}{q_0}$; and
\begin{equation}
\begin{split}
\big\|\big(\phi(t)^2-|x|^2\big)^{\gamma_1}&w\big\|_{L^2(\mathbb{R}^{1+n}_+)}\leq C
\big\|\big(\phi(t)^2-|x|^2\big)^{\gamma_2}F\big\|_{L^2(\mathbb{R}^{1+n}_+)}, \\
\end{split}
\label{equ:3.4}
\end{equation}
where $\gamma_1<\frac{m-2}{2(m+2)}$ and $\gamma_2>\frac{1}{2}$.
Since the proofs of \eqref{equ:3.3} and \eqref{equ:3.4} involve many technical analysis,
we will postpone them in next two sections respectively.
It follows from (3.3)-(3.4) and direct interpolation that (3.2) can be derived.
\end{proof}

Based on Theorem 3.1, in order to prove Theorem 1.1, we require to establish  a crucial result
as follows:

\begin{theorem}
For problem \eqref{equ:3.1}, if $F(t,x)\equiv0$ when
$|x|>\phi(t)+M-1$ and $F\in C^{\infty}([0, T_0]\times\Bbb R^n)$ for some fixed
number $T_0>0$ $(T_0<1)$, then there exist some constants $\gamma_1$ and $\gamma_2$ satisfying $0<\gamma_1<\frac{(m+2)n-2}{2(m+2)}-\frac{(m+2)n-m}{(m+2)q}$, $\gamma_2>\frac{1}{q}$, such that
\begin{equation}
\begin{split}
\Big\|\Big(\big(\phi(t)+M\big)^2-|x|^2\Big)^{\gamma_1}w\Big\|_{L^q([\f{T_0}{2}, \infty)\times \mathbb{R}^{n})}\leq C
\Big\|\Big(\big(\phi(t)+M\big)^2-|x|^2\Big)^{\gamma_2}F\Big\|_{L^{\frac{q}{q-1}}([\f{T_0}{2}, \infty)\times \mathbb{R}^{n})},
\end{split}
\label{equ:3.5}
\end{equation}
where $q>\f{2((m+2)n-m)}{(m+2)n-2}$, and $C>0$ is a positive constant  depending on $m$, $n$, $q$, $\g_1$ and $\gamma_2$.
\end{theorem}

\begin{remark}
In fact, for the application in the proof of Theorem \ref{thm:global existence} (see Section~\ref{sec:6}), it only suffices to
choose some suitable positive constants $\gamma_1$
and $\gamma_2$  such that $\gamma_1<\frac{(m+2)n-2}{2(m+2)}-\frac{(m+2)n-m}{(m+2)q}$ and $\gamma_2=(q-1)\gamma_1>\frac{1}{q}$
in the inequality \eqref{equ:3.5}.
\end{remark}

\begin{proof}
To prove \eqref{equ:3.5}, at first we focus on a special case of $F(t,x)\equiv0$ when $|x|>\phi(t)-\phi\big(\frac{T_0}{4}\big)$.
By the finite propagation speed property for the hyperbolic equations, we know that
the integral domain in  \eqref{equ:3.5} is just only  $Q=:\{(t,x): t\geq\frac{T_0}{2}, |x|\leq\phi(t)+M-1\}$.
Note that $Q$ can be covered by a finite number of curved cones $\{Q_j\}_{j=1}^{N_0}$, where the curved cone $Q_j$
$(j\ge 2)$
is a shift in the $x$ variable with respect to the curved cone
\[Q_1=\Big\{(t,x): t\ge\frac{T_0}{2}, |x|\leq\phi(t)-\phi\Big(\frac{T_0}{4}\Big)\Big\}.\]
Set
\begin{equation*}
\begin{split}
&F_1=\chi_{Q_1}F, \\
&F_2=\chi_{Q_2}(1-\chi_{Q_1})F,\\
&\qquad \qquad \ldots\\
&F_{N_0}=\chi_{Q_{N_0}}\big(1-\chi_{Q_1}-\chi_{Q_2}(1-\chi_{Q_1})-\cdot\cdot\cdot
-\chi_{Q_{N_0-1}}(1-\chi_{Q_1})\cdot\cdot\cdot(1-\chi_{Q_{N_0-2}})\big)F, \\
\end{split}
\end{equation*}
where $\chi_{Q_j}$ stands for the characteristic function of $Q_j$, and $\ds\sum_{j=1}^{N_0}F_j=F$. Let $w_j$ solve
\begin{equation*}
\begin{cases}
&\partial_t^2 w_j-t^m\triangle w_j=F_j(t,x), \\
&w_j(0,x)=0,\quad \partial_tw_j(0,x)=0.
\end{cases}
\end{equation*}
Then $suppw_j\subseteq Q_j$. Since the Tricomi equation is invariant under the translation
with respect to the variable $x$, it follows from Theorem 3.1 that
\begin{equation}\label{equ:3.6}
\big\|\big(\phi(t)^2-|x-\nu_j|^2\big)^{\gamma_1}w_j\big\|_{L^q(Q_j)}\leq C
\big\|\big(\phi(t)^2-|x-\nu_j|^2\big)^{\gamma_2}F_j\big\|_{L^{\frac{q}{q-1}}(Q_j)},
\end{equation}
where $\nu_j\in\mathbb{R}^n$ corresponds to the coordinate shift of the space variable $x$ from $Q_1$ to $Q_j$,
and $Q_j=\{(t,x): t\ge\frac{T_0}{2},
|x-\nu_j|\leq\phi(t)-\phi(\frac{T_0}{4})\}$.

Next we derive \eqref{equ:3.5} by utilizing \eqref{equ:3.6} and the condition of $t\geq\frac{T_0}{4}$.
At first, we illustrate that there exists a constant $\delta>0$ such that for $(t,x)\in Q_j$,
\begin{equation}\label{equ:3.7}
\phi(t)^2-|x-\nu_j|^2\geq\delta\Big(\big(\phi(t)+M\big)^2-|x|^2\Big).
\end{equation}
To prove \eqref{equ:3.7} for $1\le j\le N_0$, it only suffices to consider the two extreme cases: $\nu_j=0$ (corresponding to
$j=1$) and $|\nu_{j_0}|=M-1+\phi(\frac{3T_0}{8})$ (choosing  $j_0$
such that $|\nu_{j_0}|=\max_{1\le j\le N_0}|\nu_j|=M-1+\phi(\frac{3T_0}{8})$. Note that
$|\nu_{j_0}|>M-1$ holds so that the domain $Q$ can be covered by $\ds\cup_{j=1}^{N_0}Q_j$).

For $\nu_j=0$, \eqref{equ:3.7} is equivalent to
\begin{equation}\label{equ:3.8}
\phi(t)^2\geq(1-\delta)|x|^2+\delta\big(\phi(t)+M\big)^2.
\end{equation}
We now illustrate that \eqref{equ:3.8} is correct.
By $|x|\leq\phi(t)-\phi(\frac{T_0}{4})$ for $(t,x)\in Q_1$,
then in order to show \eqref{equ:3.8} it suffices to prove
\begin{equation*}
\begin{split}
\phi(t)^2\geq(1-\delta)\bigg(\phi(t)-\phi\Big(\frac{T_0}{4}\Big)\bigg)^2+\delta\big(\phi(t)+M\big)^2.
\end{split}
\end{equation*}
This is equivalent to
$$\Big\{2(1-\delta)\phi\Big(\frac{T_0}{4}\Big)-2\delta M\Big\}\phi(t)\geq(1-\delta)\phi^2\Big(\frac{T_0}{4}\Big)+\delta M^2.
$$
Obviously, this is easily achieved by $t\geq \frac{T_0}{4}$ and the smallness of $\delta$.

For $\nu_{j_0}=M-1+\phi\big(\frac{3T_0}{8}\big)$, the argument on \eqref{equ:3.7} is a little involved. First, note that
for fixed $t>0$, the domain
$Q$ is symmetric with respect to the variable $x$,
thus we can assume $\nu_{j_0}=(\nu, 0, \ldots, 0)$ with $\nu=|\nu_{j_0}|=M-1+\phi(\frac{3T_0}{8})$.
In this case, setting $x=(x_1, x')$, then \eqref{equ:3.7} is equivalent to
\begin{equation}\label{equ:3.9}
\begin{split}
\phi(t)^2&\geq|x-\nu_{j_0}|^2+\delta\big((\phi(t)+M)^2-|x|^2\big) \\
&=(1-\delta)x_1^2-2\nu x_1+\nu^2+(1-\delta)|x'|^2+\delta\big(\phi(t)+M\big)^2 \\
&=: G(t,x).
\end{split}
\end{equation}
For fixed $t>0$, $G(t,x)$ is a hyperbolic paraboloid for the variable $x$, and
takes minimum at the point $x=(\f{\nu}{1-\delta}, 0)$. Thus for the same fixed $t>0$,
the maximum of $G(t,x)$ in the domain $Q_{j_0}^t=:\{x: |x-\nu_{j_0}|\leq\phi(t)-\phi(\frac{T_0}{4})\}$
must be achieved on the boundary $\p Q_{j_0}^t=:\{x: |x-\nu_{j_0}|=\phi(t)-\phi\big(\frac{T_0}{4}\big)\}$.
Then in order to show
\eqref{equ:3.9}, our task is to prove
\begin{equation}\label{equ:3.10}
\phi(t)^2\geq\bigg(\phi(t)-\phi\Big(\frac{T_0}{4}\Big)\bigg)^2+\delta\Big(\big(\phi(t)+M\big)^2-|x|^2\Big).
\end{equation}
For this purpose, it is only enough to consider the case that $|x|^2$ takes its minimum on $\p Q_{j_0}^t$.
Note that on $\p Q_{j_0}^t$, we have
\begin{equation}\label{equ:3.11}
|x|^2=\bigg(\phi(t)-\phi\Big(\frac{T_0}{4}\Big)\bigg)^2+2\nu x_1-\nu^2.
\end{equation}
Therefore, without loss of generality, we can take
\begin{equation}\label{equ:3.12}
x_1=\nu-\phi(t)+\phi\Big(\frac{T_0}{4}\Big), \quad x'=0.
\end{equation}
Substituting \eqref{equ:3.12} and \eqref{equ:3.11} into \eqref{equ:3.10}, we are left to prove
\begin{equation}\label{equ:3.13}
\begin{split}
\phi(t)^2&\geq\bigg(\phi(t)-\phi\Big(\frac{T_0}{4}\Big)\bigg)^2+\delta\bigg\{\big(\phi(t)+M\big)^2 \\
&\quad -\bigg(\phi(t)-\phi\Big(\frac{T_0}{4}\Big)\bigg)^2+2\nu\bigg(\phi(t)-\phi\Big(\frac{T_0}{4}\Big)\bigg)-\nu^2\bigg\} \\
&=\phi(t)^2+\bigg\{2\delta\bigg(\phi\Big(\frac{T_0}{4}\Big)+M+\nu\bigg)-2\phi\Big(\frac{T_0}{4}\Big)\bigg\}\phi(t) \\
&\quad +(1-\delta)\phi\Big(\frac{T_0}{4}\Big)^2+\delta M^2-\delta\nu\bigg(\nu+2\phi\Big(\frac{T_0}{4}\Big)\bigg).
\end{split}
\end{equation}
For fixed $T_0>0$ and $M>1$, if $\delta>0$ is small enough, one then has
\begin{equation}\label{equ:3.14}
\begin{split}
&2\delta\bigg(\phi\Big(\frac{T_0}{4}\Big)+M+\nu\bigg)\leq\f{1}{2}\phi\Big(\frac{T_0}{4}\Big), \\
&(1-\delta)\phi\Big(\frac{T_0}{4}\Big)^2+\delta M^2\leq\f{3}{2}\phi\Big(\frac{T_0}{4}\Big)^2.
\end{split}
\end{equation}
By \eqref{equ:3.14} and \eqref{equ:3.13}, in order to derive \eqref{equ:3.10}, one should
derive
\[-\f{3}{2}\phi\Big(\frac{T_0}{4}\Big)\phi(t)+\f{3}{2}\phi^2\Big(\frac{T_0}{4}\Big)\leq0.\]
Obviously, this holds true by $t\geq\frac{T_0}{4}$. Then \eqref{equ:3.10} is proved.

Consequently, for $(t,x)\in\bigcup_{j=1}^{N_0}Q_j$,
there exists a fixed positive constant $c>0$ such that for $1\le j\le N_0$,
\begin{equation}\label{equ:3.15}
\begin{split}
c\Big(\big(\phi(t)+M\big)^2-|x|^2\Big)\leq\phi(t)^2-|x-\nu_j|^2.
\end{split}
\end{equation}
Note that by $|x|\leq\phi(t)+M-1$ for $(t,x)\in Q$, one has
\begin{equation}\label{equ:3.16}
\begin{split}
&2\big\{\big(\phi(t)+M\big)^2-|x|^2\big\}-\{\phi^2(t)-|x-\nu_j|^2\} \\
&\geq (|x|+1)^2-|x|^2+\big(\phi(t)+M\big)^2-|x|^2-\phi(t)^2+|x-\nu_j|^2 \\
&=2M\phi(t)+M^2+|\nu_j|^2+1+2(1-|\nu_j|)|x|.
\end{split}
\end{equation}
If $1-|\nu_j|<0$, then by $|\nu_j|\leq M-1+\phi\big(\frac{3T_0}{8}\big)$ and the smallness of
$T_0$, the last line in \eqref{equ:3.16} is bounded from below by
\begin{equation}\label{equ:3.17}
\begin{split}
&2M\phi(t)+M^2+|\nu_j|^2+1+2\Big\{2-M-\phi\Big(\frac{3T_0}{8}\Big)\Big\}\{\phi(t)+M-1\} \\
&=4\phi(t)-M^2+6M-3+|\nu_j|^2-2\phi\Big(\frac{3T_0}{8}\Big)\phi(t)-2(M-1)\phi\Big(\frac{3T_0}{8}\Big) \\
&\geq 2\phi(t)-M^2+1;
\end{split}
\end{equation}
while in the case of $1-|\nu_j|\geq0$, it follows from \eqref{equ:3.16}  that
\begin{equation}\label{equ:3.18}
2\big\{\big(\phi(t)+M\big)^2-|x|^2\big\}-\{\phi(t)^2-|x-\nu_j|^2\}\geq M^2+1>0.
\end{equation}
Substituting \eqref{equ:3.17}-\eqref{equ:3.18} into \eqref{equ:3.16} yields that for $2\phi(t)\geq M^2-1$,
\begin{equation}\label{equ:3.19}
\begin{split}
\phi(t)^2-|x-\nu_j|^2\leq C \Big(\big(\phi(t)+M\big)^2-|x|^2\Big).
\end{split}
\end{equation}
On the other hand, if $2\phi(t)< M^2-1$, then
\begin{equation}\label{equ:3.20}
\begin{split}
\phi(t)^2-|x-\nu_j|^2\leq\phi(t)^2\leq C_M\leq C_M\Big(\big(\phi(t)+M\big)^2-|x|^2\Big).
\end{split}
\end{equation}
Thus it follows from \eqref{equ:3.19}-\eqref{equ:3.20} that \eqref{equ:3.15} holds.
Therefore,
\begin{equation*}
\begin{split}
&\Big\|\Big(\big(\phi(t)+M\big)^2-|x|^2\Big)^{\gamma_1}w\Big\|_{L^q([\frac{T_0}{2},\infty)\times\mathbb{R}^n)}\\
&\leq C\sum_{j=1}^{N_0}\Big\|\Big(\big(\phi(t)+M\big)^2-|x|^2\Big)^{\gamma_1}w_j\Big\|_{L^q(Q_j)} \\
&\leq C\sum_{j=1}^{N_0}\big\|\big(\phi(t)^2-|x-\nu_j|^2\big)^{\gamma_1}w_j\big\|_{L^q(Q_j)}\qquad\qquad \text{\big(by \eqref{equ:3.7}\big)} \\
&\leq C\sum_{j=1}^{N_0}\big\|\big(\phi(t)^2-|x-\nu_j|^2\big)^{\gamma_2}F_j\big\|_{L^{\frac{q}{q-1}}(Q_j)}\qquad\quad \text{\big(by \eqref{equ:3.6}\big)}  \\
&\leq C\sum_{j=1}^{N_0}\Big\|\Big(\big(\phi(t)+M\big)^2-|x|^2\Big)^{\gamma_2}F_j\Big\|_{L^{\frac{q}{q-1}}(Q_j)} \qquad \text{\big(by \eqref{equ:3.15}\big)} \\
&\leq C_{N_0}\Big\|\Big(\big(\phi(t)+M)^2-|x|^2\big)^{\gamma_2}F\Big\|_{L^{\frac{q}{q-1}}([\frac{T_0}{2},\infty)\times\mathbb{R}^n)},
\end{split}
\end{equation*}
which derives \eqref{equ:3.5}.
\end{proof}

\section{The proof of (3.3)}\label{sec4}
One can write inequality \eqref{equ:3.3} as
\begin{align}\label{equ:4.1}
\left\|\lc\phi(t)^2-|x|^2\rc^{\frac{1}{q_0}-\nu}w\right\|_{L^{q_0}(\mathbb{R}^{1+n}_+)}
\leq C\left\|\lc\phi(t)^2-|x|^2\rc^{\frac{1}{q_0}+\nu}F\right\|_{L^{\frac{q_0}{q_0-1}}(\mathbb{R}^{1+n}_+)},
\end{align}
where $\nu>0$, and $q_0=\frac{2((m+2)n+2)}{(m+2)n-2}$.
To prove \eqref{equ:4.1}, it suffices to prove the following inequality for $T\geq\bar{T}$,
\begin{equation}
\begin{split}
\Big\|\big(\phi(t)^2-|x|^2\big)^{\frac{1}{q_0}-\nu}&w\Big\|_{L^{q_0}([\frac{T}{2}, T]\times\mathbb{R}^n)} \leq C\phi(T)^{-\nu}|\ln{T}|^{\frac{1}{q_0}}\Big\|\big(\phi(t)^2-|x|^2\big)^{\frac{1}{q_0}+\nu}F\Big\|_{L^{\frac{q_0}{q_0-1}}
(\mathbb{R}^{1+n}_+)},
\end{split}
\label{equ:4.2}
\end{equation}
where $\bar{T}>0$ is a fixed large constant.
Indeed, Lemma 3.4 of \cite{HWYin} implies
\begin{equation}
\begin{split}
\parallel w\parallel_{L^{q_0}([0, \bar{T}/2]\times\mathbb{R}^n)}
\leq C\parallel F\parallel_{L^{\frac{q_0}{q_0-1}}([0, \bar{T}/2]\times\mathbb{R}^n)}.
\end{split}
\label{equ:H.0}
\end{equation}
When \(0<t\leq \bar{T}\) and \(\phi(t)-|x|\geq1\), the weight function $\phi(t)^2-|x|^2$
in the inequality \eqref{equ:4.2} is bounded from below and above, thus we have
that from \eqref{equ:H.0}
\begin{equation}\label{equ:4.102}
\begin{split}
\parallel(\phi(t)^2-|x|^2)^{\frac{1}{q_0}-\nu}w\parallel_{L^{q_0}([0, \bar{T}/2]\times\mathbb{R}^n)}
\leq C\parallel(\phi(t)^2-|x|^2)^{\frac{1}{q_0}+\nu}F\parallel_{L^{\frac{q_0}{q_0-1}}([0, \bar{T}/2]\times\mathbb{R}^n)}.
\end{split}
\end{equation}
Then summing \eqref{equ:4.2} over all the $T\geq\bar{T}$ together with \eqref{equ:4.102} yields \eqref{equ:4.1}.

Note that $F(t,x)\equiv0$ for $|x|>\phi(t)-1$, then this means
$\operatorname{supp}\ F\subseteq\{(t,x): |x|^2\leq\phi(t)^2-1\}$.
Set $F=F^0+F^1$, where
\begin{equation}
F^0=\left\{ \enspace
\begin{aligned}
&F,  &&\phi(t)\leq\frac{\phi(1)\phi(T)}{10\phi(2)}, \\
&0, &&\phi(t)>\frac{\phi(T)\phi(1)}{10\phi(2)}. \\
\end{aligned}
\right.
\label{equ:4.30}
\end{equation}
Correspondingly, $w=w^0+w^1$, and $w^j$ $(j=0,1)$ solves
\begin{equation*}
\left\{ \enspace
\begin{aligned}
&\partial_t^2w^j-t^m\Delta w^j=F^j \\
&w^j(0,x)=0, \quad \partial_t w^j(0,x)=0.
\end{aligned}
\right.
\end{equation*}
To prove \eqref{equ:4.2}, it suffices to show that for \(j=0,1\),
\begin{equation}
\begin{split}
\Big\|&\big(\phi(t)^2-|x|^2\big)^{\frac{1}{q_0}-\nu}w^j\Big\|_{L^{q_0}(\{(t,x):\frac{T}{2}\leq t\leq T\})}
\leq C\phi(T)^{-\frac{\nu}{4}}\Big\|\big(\phi(t)^2-|x|^2\big)^{\frac{1}{q_0}+\nu}F^j\Big\|_{L^{\frac{q_0}{q_0-1}}(\mathbb{R}^{1+n}_+)}.\\
\end{split}
\label{equ:4.34}
\end{equation}
For this purpose, we shall make some reductions.
By making these reductions, we shall restrict the support of \(F_j\) and \(w_j\)
in certain domains, such that in each domain the characteristic weight $\phi(t)^2-|x|^2$
on both sides of \eqref{equ:4.34} can be
removed. In this case, our task is reduced to prove some unweighted Strichartz estimates,
which is relatively easier than to prove \eqref{equ:4.34} directly by applying some
basic techniques from Fourier integral operators.

Next we give a precise description on the reductions. First, we suppose $\operatorname{supp}F^j\subseteq [T_0, 2T_0]\times\mathbb{R}^n$ for some fixed
constant $T_0>0$ satisfying $\phi(2T_0)\geq1$. Then \eqref{equ:4.34} follows from
\begin{equation}
\begin{split}
\Big\|&\big(\phi(t)^2-|x|^2\big)^{\frac{1}{q_0}-\nu}w^j\Big\|_{L^{q_0}(\{(t,x):\frac{T}{2}\leq t\leq T\})} \\
&\leq C\phi(T_0)^{-\frac{\nu}{4}}\phi(T)^{-\frac{\nu}{4}}\Big\|\big(\phi(t)^2-
|x|^2\big)^{\frac{1}{q_0}+\nu}F^j\Big\|_{L^{\frac{q_0}{q_0-1}}(\mathbb{R}^{1+n}_+)}\\
\end{split}
\label{equ:4.35}
\end{equation}
by summing up all these $T_0$.

Second, we pose more restrictions on the support of \(F_j\). That is, we assume that $F^j\equiv0$ holds
when the value of $\phi(t)-|x|$ is not in $[\delta_0\phi(T_0), 2\delta_0\phi(T_0)]$ for some fixed
constant $\dl_0$ with $0<\delta_0\le 2$ and $\delta_0\phi(T_0)\geq1$, then in order to prove \eqref{equ:4.35} we only need to show
\begin{equation}
\begin{split}
\Big\|&\big(\phi(t)^2-|x|^2\big)^{\frac{1}{q_0}-\nu}w^j\Big\|_{L^{q_0}(\{(t,x):\frac{T}{2}\leq t\leq T\})} \\
&\leq C\phi(T_0)^{-\frac{\nu}{2}}\phi(T)^{-\frac{\nu}{4}}
\Big\|\big(\phi(t)^2-|x|^2\big)^{\frac{1}{q_0}+\nu}F^j\Big\|_{L^{\frac{q_0}{q_0-1}}(\mathbb{R}^{1+n}_+)}.\\
\end{split}
\label{equ:4.36}
\end{equation}

Third, we make a dyadic decomposition on the variable $\phi(t)-|x|$ in the support of $w^j$, such that
in order to prove \eqref{equ:4.36} it suffices to show that for $\delta\geq\delta_0$,
\begin{equation}
\begin{split}
\Big\|&\big(\phi^2(t)-|x|^2\big)^{\frac{1}{q_0}-\nu}w^j\Big\|_{L^{q_0}(\{(t,x):\frac{T}{2}\leq t\leq T, \delta\phi(T_0)\leq\phi(t)-|x|\leq
 2\delta\phi(T_0)\})} \\
&\leq C(\phi(T_0)\phi(T))^{-\frac{\nu}{2}}
\Big\|\big(\phi(t)^2-|x|^2\big)^{\frac{1}{q_0}+\nu}F^j\Big\|_{L^{\frac{q_0}{q_0-1}}(\mathbb{R}^{1+n}_+)}.
\end{split}
\label{equ:4.37}
\end{equation}
With these reductions, to prove \eqref{equ:4.37}, the left task is to establish
\begin{align}\label{equ:4.A}
(\phi(T_0)\phi(T)\delta)^{\frac{1}{q_0}-\nu}\parallel& w^j
\parallel_{L^{q_0}(\{(t,x):\frac{T}{2}\leq t\leq T, \delta\phi(T_0)\leq\phi(t)-|x|\leq
 2\delta\phi(T_0)\})} \no\\
\leq C&(\phi(T_0)\phi(T))^{-\frac{\nu}{2}}(\phi^2(T_0)\delta_0)^{\frac{1}{q_0}+\nu}
\parallel F^j\parallel_{L^{\frac{q_0}{q_0-1}}(\mathbb{R}^{1+n}_+)}.
\end{align}
By rearranging some terms in \eqref{equ:4.A}, then \eqref{equ:4.A} directly follows from
\begin{align}\label{equ:4.B}
\Big(\frac{\phi(T)}{\phi(T_0)}\Big)^{\frac{1}{q_0}-\frac{\nu}{2}}\delta^{\frac{1}{q_0}+\frac{\nu}{2}}\phi(T_0)^{-3\nu}
\delta^{-\f{3}{2}\nu}\delta_0^{-\nu}
\parallel& w^j\parallel_{L^{q_0}(\{(t,x):\frac{T}{2}\leq t\leq T, \delta\phi(T_0)\leq\phi(t)-|x|\leq2\delta\phi(T_0)\})} \no\\
&\leq C\delta_0^{\frac{1}{q_0}}\parallel F^j\parallel_{L^{\frac{q_0}{q_0-1}}(\mathbb{R}^{1+n}_+)}.
\end{align}
Note that
\[\phi(T_0)^{-3\nu}\delta^{-\f{3}{2}\nu}\delta_0^{-\nu}\leq\phi(T_0)^{-3\nu}\delta_0^{-\f{5}{2}\nu}
=(\phi(T_0)\delta_0)^{-3\nu}\delta_0^{\f{\nu}{2}}\leq\delta_0^{\f{\nu}{2}}\leq2^{\f{\nu}{2}}.\]
Therefore, \eqref{equ:4.B} follows from
\begin{equation}\label{equ:4.38}
\begin{split}
\Big(\frac{\phi(T)}{\phi(T_0)}\Big)^{\frac{1}{q_0}-\frac{\nu}{2}}\delta^{\frac{1}{q_0}+\frac{\nu}{2}}
\parallel& w^j\parallel_{L^{q_0}(\{(t,x):\frac{T}{2}\leq t\leq T, \delta\phi(T_0)\leq\phi(t)-|x|\leq2\delta\phi(T_0)\})} \\
&\leq C\delta_0^{\frac{1}{q_0}}\parallel F^j\parallel_{L^{\frac{q_0}{q_0-1}}(\mathbb{R}^{n+1}_+)}.
\end{split}
\end{equation}
We next start to prove \eqref{equ:4.38}. Set $G^j(t,x)=:T_0^2F^j(T_0t, T_0^{\frac{m+2}{2}}x)$ and
$v^j(t,x)=:w^j(T_0t, T_0^{\frac{m+2}{2}}x)$ for \(j=0,1\).
Then $v^j$ satisfies
\begin{equation}\label{equ:4.104}
\begin{cases}
&\partial_t^2 v^j-t^m\triangle v^j=G^j(t,x), \\
&v^j(0,x)=0,\quad \partial_tv^j(0,x)=0, \\
\end{cases}
\end{equation}
where $\operatorname{supp}G^j\subseteq\{(t,x):1\leq t\leq2,\delta_0\leq \phi(t)-|x|\leq2\delta_0\}$. Then, if we let $T$ denote
by $\frac{T}{T_0}$, then \eqref{equ:4.38} is a result of
\begin{equation}
\phi(T)^{\frac{1}{q_0}-\frac{\nu}{2}}\delta^{\frac{1}{q_0}+\frac{\nu}{2}}
\parallel v^j\parallel_{L^{q_0}(\{(t,x):\frac{T}{2}\leq t\leq T, \delta\leq\phi(t)-|x|\leq2\delta\})}
\leq C\delta_0^{\frac{1}{q_0}}\parallel G^j\parallel_{L^{\frac{q_0}{q_0-1}}(\mathbb{R}^{1+n}_+)}.
\label{equ:4.85}
\end{equation}

After this reduction, by \eqref{equ:4.30}, we have that \(T/T_0\geq \phi^{-1}\big(10\phi(2)\big)\)
holds for $(t,x)\in \operatorname{supp}w^0$ and \(T/2\leq t\leq T\), or equivalently, \(T\geq \phi^{-1}\big(10\phi(2)\big)\)
holds for $(t,x)\in \operatorname{supp}v^0$ and \(T/2\leq t\leq T\), which is called the relatively ``large time". On the other hand,
\(1\leq T/T_0\leq \phi^{-1}\big(10\phi(2)\big)\) holds  for $(t,x)\in \operatorname{supp}w^1$ and \(\f{T}{2}\leq t\leq T\),
or equivalently, \(1\leq T\leq \phi^{-1}\big(10\phi(2)\big)\) holds  for $(t,x)\in \operatorname{supp}v^1$ and
\(\f{T}{2}\leq t\leq T\), which is called
the relatively ``small time". In Subsection 4.1 and Subsection 4.2, we shall handle the two cases
respectively. For the concision of notation, in the following subsections, we will omit the superscript \(j\) and prove
\begin{equation}
\phi(T)^{\frac{1}{q_0}-\frac{\nu}{2}}\delta^{\frac{1}{q_0}+\frac{\nu}{2}}
\parallel v\parallel_{L^{q_0}(\{(t,x):\frac{T}{2}\leq t\leq T, \delta\leq\phi(t)-|x|\leq2\delta\})}
\leq C\delta_0^{\frac{1}{q_0}}\parallel G\parallel_{L^{\frac{q_0}{q_0-1}}(\mathbb{R}^{1+n}_+)}.
\label{equ:4.39}
\end{equation}

\subsection{Estimate for large time}\label{sec4:large}
We first deal with the case of ``relative large time" in \eqref{equ:4.39}, which means $\phi(T)\geq10\phi(2)$.

The proof of \eqref{equ:4.39} (corresponding to the estimate of $w^0$) can be divided into the following three parts according to
the different values of $\delta/\delta_0$:

Case (i) $\delta_0\leq\delta\leq40\cdot\frac{\phi(2)}{\phi(1)}\delta_0$;

Case (ii) $\delta\geq10\phi(2)$;

Case (iii) $40\cdot\frac{\phi(2)}{\phi(1)}\delta_0\leq\delta\leq10\phi(2), \delta_0<\f{1}{4}\phi(1)$.

\begin{remark}
We add the condition \(\delta_0<\f{1}{4}\phi(1)\) in Case (iii) for the reason
that \(40\cdot\frac{\phi(2)}{\phi(1)}\delta_0<10\phi(2)\) makes sense
only for \(\delta_0<\f{1}{4}\phi(1)\).
\end{remark}

Here we point out that for the wave equation and Case (i)- Case (ii), it is rather direct to
establish the analogous inequality \eqref{equ:4.39} (see (3.2) and $\S 3$ of \cite{Gls}). However, for the Tricomi equation,
due to the complexity of its
fundamental solution and degeneracy,
it needs more delicate and involved techniques from the knowledge of microlocal analysis to get the pointwise estimate
of $v$. For the proof of \eqref{equ:4.39} in case (iii), we shall utilize some basic ideas in \cite{Gls}.

\subsubsection{The proof of \eqref{equ:4.39} in Case (i)}\label{sec4:large:i}

Note that $\phi(T)>\phi(T_0)\geq1$ and $\delta\phi(T_0)\leq\phi(T)$. To prove \eqref{equ:4.39}, it suffices to show
\begin{equation}
\phi(T)^{\frac{1}{q_0}}\parallel v\parallel_{L^{q_0}(\{(t,x):\frac{T}{2}\leq t\leq T, \delta\leq\phi(t)-|x|\leq2\delta\})}\leq C\parallel G\parallel_{L^{\frac{q_0}{q_0-1}}(\mathbb{R}^{1+n}_+)}.
\label{equ:4.40}
\end{equation}
Use the method in Lemma 3.3 of \cite{HWYin}, if we write
\begin{equation}\label{equ:4.78}
v(t,x)=\int_0^tH(t,s,x)\md s,
\end{equation}
we then arrive at

\begin{claim}
\begin{equation}\label{equ:4.41}
\parallel H(t,s,\cdot)\parallel_{L^{q_0}(\mathbb{R}^n)}\leq C|t-s|^{-\frac{2}{q_0}(1+\frac{m}{4})}\parallel G(s,\cdot)\parallel_{L^{\frac{q_0}{q_0-1}}(\mathbb{R}^n)}.
\end{equation}
\end{claim}
\begin{proof}[Proof of claim]
If we repeat the reduction of (3.23)-(3.24) in \cite{HWYin}, we then have $H=\ds\sum_{j=-\infty}^{\infty}H_j$, where
\begin{equation}\label{equ:4.79}
H_j=:T_jG(t,s,x)=\int_{\mathbb{R}^n}e^{i(x\cdot\xi+(\phi(t)-\phi(s))|\xi|)}
\beta(\frac{|\xi|}{2^j})a(t,s,\xi)\hat{G}(s,\xi) \md\xi,
\end{equation}
where the cut-off function \(\beta\) was defined in \eqref{equ:2.13} and the amplitude function $a$ satisfies
\begin{equation}\label{equ:4.80}
\big| \partial_\xi^\beta a(t,s,\xi)\big|\leq C\big(1+\phi(t)|\xi|\big)^{-\frac{m}{2(m+2)}}\big(1+\phi(s)|\xi|\big)^{-\frac{m}{2(m+2)}}|\xi|^{-\frac{2}{m+2}-|\beta|}.
\end{equation}
If we further set
\[\bar{T}_jG(t,s,x)=:\lambda_j^{\frac{2m}{(m+2)((m+2)n+2)}}H_j(t,s,x)\quad\text{with $\lambda_j=2^j$},\]
and  repeat the computation of (3.29) and (3.30) in \cite{HWYin}, we then obtain
\begin{equation*}
%\label{equ:4.42}
\begin{split}
\parallel \bar{T}_jG (t,s,\cdot)&\parallel_{L^2(\mathbb{R}^n)} \leq \lambda_j^{-\frac{2}{m+2}}C\big(1+\lambda_j\phi(t)\big)^{-\frac{m}{2(m+2)}\cdot\frac{(m+2)n-2}{(m+2)n+2}} \\
&\qquad\qquad\quad  \times |t-s|^{-\frac{m}{(m+2)n+2}}\parallel G(s,\cdot)\parallel_{L^2(\mathbb{R}^n)}
\end{split}
\end{equation*}
and
\begin{equation*}
\begin{split}
%\label{equ:4.43}
\qquad\qquad\quad\parallel \bar{T}_jG (t,s,\cdot)&\parallel_{L^\infty(\mathbb{R}^n)}\leq C\lambda_j^{\frac{n+1}{2}-\frac{2}{m+2}}\big(1+\lambda_j\phi(t)\big)^{-\frac{m}{2(m+2)}\cdot\frac{(m+2)n-2}{(m+2)n+2}} \\
&\qquad\qquad\quad  \times |t-s|^{-\frac{m}{(m+2)n+2}}|t-s|^{-\frac{n-1}{2}\cdot\frac{m+2}{2}}\parallel
G(s,\cdot)\parallel_{L^1(\mathbb{R}^n)}.
\end{split}
\end{equation*}
Therefore, if $j\geq0$, then
\begin{align*}
\parallel \bar{T}_jG (t,s,\cdot)\parallel_{L^2(\mathbb{R}^n)} &\leq C\lambda_j^{-\frac{2}{m+2}}\big(1+\phi(t)\big)^{-\frac{m}{2(m+2)}\cdot\frac{(m+2)n-2}{(m+2)n+2}}
|t-s|^{-\frac{m}{(m+2)n+2}}\parallel G(s,\cdot)\parallel_{L^2(\mathbb{R}^n)} \\
&\leq C\lambda_j^{-\frac{2}{m+2}}(t-s)^{-\frac{m}{4}\cdot\frac{(m+2)n-2}{(m+2)n+2}}|t-s|^{-\frac{m}{(m+2)n+2}}\parallel G(s,\cdot)\parallel_{L^2(\mathbb{R}^n)}
\end{align*}
and
\begin{align*}
\parallel &\bar{T}_jG (t,s,\cdot)\parallel_{L^\infty(\mathbb{R}^n)} \\
&\leq C\lambda_j^{\frac{n+1}{2}-\frac{2}{m+2}}(t-s)^{-\frac{m}{4}\cdot\frac{(m+2)n-2}{(m+2)n+2}}|t-s|^{-\frac{n-1}{2}\cdot\frac{m+2}{2}}
|t-s|^{-\frac{m}{(m+2)n+2}}\parallel G(s,\cdot)\parallel_{L^1(\mathbb{R}^n)}.
\end{align*}
Using the interpolation and direct computation, we have that for $j\geq0$,
\begin{equation}\label{equ:4.44}
\parallel H_j(t,s,\cdot)\parallel_{L^{q_0}(\mathbb{R}^n)}\leq C(t-s)^{-(\frac{m}{4}+1)\cdot\frac{2}{q_0}}\parallel G(s,\cdot)\parallel_{L^{\frac{q_0}{q_0-1}}(\mathbb{R}^n)}.
\end{equation}
For $j<0$, let
\[T_-G (t,s,x)\equiv\ds\sum_{j<0}H_j(t,s,x)=\int_{|\xi|\leq1}e^{i(x\cdot\xi+(\phi(t)-\phi(s))|\xi|)}a(t,s,\xi)\hat{G}(s,\xi) d\xi.\]
Then it follows from Plancherel's identity that
\begin{equation*}
\begin{split}
\parallel T_-G& (t,s,\cdot)\parallel_{L^2(\mathbb{R}^n)}
\leq C\Big(\int_{|\xi|\leq1}\left||\xi|^{-\f{2}{m+2}}\big(1+\phi(t)|\xi|\big)^{-\frac{m}{2(m+2)}}\hat{G}(s,\xi)\right|^2 \md\xi\Big)^{\frac{1}{2}}.
\end{split}
\end{equation*}
Note that
\[|\hat{G}(s,\xi)|=|\int_{\mathbb{R}^n}e^{-iy\cdot\xi}G(s,y)\md y|=|\int_{|y|\leq\phi(2)}e^{-iy\cdot\xi}G(s,y)\md y|\leq C\parallel G(s,\cdot)\parallel_{L^2(\mathbb{R}^n)},\]
and
\begin{equation*}%\label{equ:4.45}
\begin{split}
&\Big(\int_{|\xi|\leq1}\big||\xi|^{-\f{2}{m+2}}(1+\phi(t)|\xi|)^{-\frac{m}{2(m+2)}}\big|^2 \md\xi\Big)^{\frac{1}{2}}\\
&=C\Big(\int_0^1\big|r^{-\f{2}{m+2}}(1+\phi(t)r)^{-\frac{m}{2(m+2)}}\big|^2 r^{n-1}\md r\Big)^{\frac{1}{2}} \\
&\le C\Big(\int_0^1r^{n-1-\frac{m+4}{m+2}}\phi(t)^{-\frac{m}{m+2}
}\md r\Big)^{\frac{1}{2}}\\
&\le C\phi(t)^{-\frac{m}{2(m+2)}},
\end{split}
\end{equation*}
here we have used the fact of $n-1-\frac{m+4}{m+2}\geq-\frac{2}{m+2}>-1$ for $n\geq3$ and $m\geq1$.
Thus,
\begin{equation*}%\label{equ:4.46}
\parallel T_-G (t,s,\cdot)\parallel_{L^2(\mathbb{R}^n)}\leq C\phi(t)^{-\frac{m}{2(m+2)}}\parallel G(s,\cdot)\parallel_{L^2(\mathbb{R}^n)}
\leq C|t-s|^{-\frac{m}{4}}\parallel G(s,\cdot)\parallel_{L^2(\mathbb{R}^n)}.
\end{equation*}
Similarly, one has
\begin{equation*}%\label{equ:4.47}
%\begin{split}
\parallel T_-G (t,s,\cdot)\parallel_{L^\infty(\mathbb{R}^n)}
\leq C|t-s|^{-\frac{m}{4}-\frac{n-1}{2}\cdot\frac{m+2}{2}}\parallel G(s,\cdot)\parallel_{L^1(\mathbb{R}^n)}.
%\end{split}
\end{equation*}
Using interpolation again, we get
\begin{equation}\label{equ:4.48}
\parallel T_-G (t,s,\cdot)\parallel_{L^{q_0}(\mathbb{R}^n)}\leq C|t-s|^{-\frac{2}{q_0}(1+\frac{m}{4})}\parallel G(s,\cdot)\parallel_{L^{\frac{q_0}{q_0-1}}(\mathbb{R}^n)}.
\end{equation}
Then by Littlewood-Paley theory,  \eqref{equ:4.44} and \eqref{equ:4.48}, \eqref{equ:4.41}  in {\bf Claim 4.1}
is derived.
\end{proof}
With {\bf Claim 4.1}, noting that $t\lesssim t-s$ holds on the support of $G(s,x)$, then
we have
\begin{equation*}%\label{equ:4.49}
\begin{split}
\parallel v\parallel_{L^{q_0}([\f{T}{2},T]\times\mathbb{R}^n)}\leq&\Big\|\int_0^t\parallel u(t,s,\cdot)\parallel_{L^{q_0}(\mathbb{R}^n)}\md s\Big\|_{L^{q_0}_t([\f{T}{2},T])} \\
\leq&C\Big\|\int_0^t|t-s|^{-\frac{2}{q_0}(1+\frac{m}{4})}\parallel G(s,\cdot)\parallel_{L^{\frac{q_0}{q_0-1}}_x}\md s\Big\|_{L^{q_0}_t} \\
\leq&C\Big\||t|^{-\frac{2}{q_0}(1+\frac{m}{4})}\int_1^2\parallel G(s,\cdot)\parallel_{L^{\frac{q_0}{q_0-1}}_x}\md s\Big\|_{L^{q_0}_t} \\
\leq&C\Big(\int_\frac{T}{2}^Tt^{(-\frac{2}{q_0}(1+\frac{m}{4}))q_0}\md t\Big)^{\frac{1}{q_0}}\parallel G\parallel_{L^{\frac{q_0}{q_0-1}}(\mathbb{R}^{n+1}_+)} \quad \text{(H\"{o}lder's inequality)}\\
\leq&CT^{-\frac{1}{q_0}\frac{m+2}{2}}\parallel G\parallel_{L^{\frac{q_0}{q_0-1}}(\mathbb{R}^{n+1}_+)} \\
\leq&C\phi(T)^{-\frac{1}{q_0}}\parallel G\parallel_{L^{\frac{q_0}{q_0-1}}(\mathbb{R}^{n+1}_+)},
\end{split}
\end{equation*}
which derives \eqref{equ:4.40}.

\subsubsection{The proof of \eqref{equ:4.39} in Case (ii)}\label{sec4:large:ii}

In this case, we have $\phi(t)-|x|\geq\delta\geq10\phi(2)$. As in \eqref{equ:4.78} and \eqref{equ:4.79}-\eqref{equ:4.80}, we can write
\[v=\sum_{j=-\infty}^{\infty}v_j=\sum_{j=-\infty}^{\infty}\int_0^t\int_{\Bbb R^n} K_j(t,x;s,y)G(s,y)dyds,\]
where
\begin{equation}
 K_j(t,x;s,y)=\int_{\mathbb{R}^n}e^{i((x-y)\cdot\xi+(\phi(t)-\phi(s))|\xi|)}
\beta(\frac{|\xi|}{2^j})a(t,s,\xi)\hat{G}(s,\xi) d\xi,
\label{equ:4.50}
\end{equation}
moreover, as in (3.41) of \cite{HWYin} for $\lambda_j=2^j$ and any $N\in\Bbb R^+$,
\begin{equation}
\begin{split}
|K_j(t,x;s,y)|\leq &C_N\lambda_j^{\frac{n+1}{2}-\frac{2}{m+2}}\big(|\phi(t)-\phi(s)|
+\lambda_j^{-1}\big)^{-\frac{n-1}{2}}\big(1+\phi(t)\lambda_j\big)^{-\frac{m}{2(m+2)}} \\
&\times \Big(1+\lambda_j\big||\phi(t)-\phi(s)|-|x-y|\big|\Big)^{-N}.
\end{split}
\label{equ:4.51}
\end{equation}
Denote $D_{s,y}=\{(s,y): 1\leq s\leq2, \phi(s)-2\delta_0\leq|y|\leq\phi(s)-\delta_0\}$. By H\"{o}lder's inequality
and the compact support property of $G(s,y)$ with respect to the variable $(s,y)$, we arrive at
\[|v_j|\leq\Big\| K_j(t,x;s,y)\big(\phi(t)^2-|x|^2\big)^{\frac{1}{q_0}}
\Big\|_{L^{q_0}(D_{s,y})}\Big\|\big(\phi(t)^2-|x|^2\big)^{-\frac{1}{q_0}}G(s,y)
\Big\|_{L^{\frac{q_0}{q_0-1}}(D_{s,y})}.
\]
In addition, applying the compact support property of $G(s,y)$, it is easy to check
\[\phi(t)-\phi(s)-|x-y|\geq C(\phi(t)-|x|), \quad \phi(t)-\phi(s)+|x-y|\sim\phi(t).\]
Based on this, let $N=\frac{n}{2}-\frac{1}{m+2}$ in \eqref{equ:4.51}, we then have
\begin{equation}\label{equ:4.52}
\begin{split}
&\Big\| K_j(t,x;s,y)\big(\phi(t)^2-|x|^2\big)^{\frac{1}{q_0}}\Big\|_{L^{q_0}(D_{s,y})} \\
&\leq C\bigg(\int_1^2\int_{\phi(s)-2\delta_0\leq|y|\leq\phi(s)-\delta_0}\{\lambda_j^{\frac{n+1}{2}-\frac{2}{m+2}}\big(|\phi(t)-\phi(s)|
+\lambda_j^{-1}\big)^{-\frac{n-1}{2}}(1+\phi(t)\lambda_j)^{-\frac{m}{2(m+2)}} \\
&\qquad \times \Big(1+\lambda_j\big||\phi(t)-\phi(s)|-|x-y|\big|\Big)^{\frac{1}{m+2}-\frac{n}{2}}\}^{q_0}
\phi(t)(\phi(t)-|x|)\md y\md s\bigg)^{\frac{1}{q_0}} \\
&\leq C\phi(t)^{-\frac{n-1}{2}+\frac{1}{q_0}-\frac{m}{2(m+2)}}\big(\phi(t)-|x|\big)^{\frac{1}{m+2}-\frac{n}{2}+\frac{1}{q_0}}
\Big(\int_1^2\int_{\phi(s)-2\delta_0\leq|y|\leq \phi(s)-\delta_0}\md y\md s\Big)^{\frac{1}{q_0}} \\
&\leq C\delta_0^{\frac{1}{q_0}}\phi(t)^{-\frac{n-1}{2}+\frac{1}{q_0}-\frac{m}{2(m+2)}}\big(\phi(t)-|x|\big)^{\frac{1}{m+2}-\frac{n}{2}+\frac{1}{q_0}}
\end{split}
\end{equation}
and
\begin{equation}\label{equ:4.53}
\begin{split}
&\Big(\int_1^2\int_{\phi(s)-2\delta_0\leq|y|\leq \phi(s)-\delta_0}\big\{(\phi(t)^2-|x|^2)^{-\frac{1}{q_0}}G(s,y)\big\}^{\frac{q_0}{q_0-1}}\md y\md s\Big)^{\frac{q_0-1}{q_0}} \\
&\leq C\big(\delta\phi(T_0)\phi(T)\big)^{-\frac{1}{q_0}}\parallel G\parallel_{L^{\frac{q_0}{q_0-1}}(\mathbb{R}^{1+n}_+)} \\
&\leq C\big(\delta\phi(T)\big)^{-\frac{1}{q_0}}\parallel G\parallel_{L^{\frac{q_0}{q_0-1}}(\mathbb{R}^{1+n}_+)}.
\end{split}
\end{equation}
On the other hand,
\begin{equation}
\begin{split}
&\Big\|\phi(t)^{-\frac{n-1}{2}+\frac{1}{q_0}-\frac{m}{2(m+2)}}\big(\phi(t)-|x|\big)^{\frac{1}{m+2}-\frac{n}{2}
+\frac{1}{q_0}}\Big\|_{L^{q_0}([\frac{T}{2}, T]\times\mathbb{R}^n)}\\
&=C_n\Big(\int_{\frac{T}{2}}^T\phi(t)^{-\frac{q_0}{2}(n-\frac{2}{m+2})+1}\int_0^{\phi(t)-10\phi(2)}
\big(\phi(t)-r\big)^{q_0(\frac{1}{m+2}-\frac{n}{2})+1}r^{n-1}\md r\md t\Big)^{\frac{1}{q_0}} \\
&\leq C\Big(\int_{\frac{T}{2}}^T\phi(t)^{-\frac{2}{m+2}}\md t\Big)^{\frac{1}{q_0}} \\
&\leq C.
\end{split}
\label{equ:4.54}
\end{equation}
Therefore, combining \eqref{equ:4.52}-\eqref{equ:4.54} yields
\newpage
\begin{align}\label{equ:4.103}
\parallel v_j\parallel_{L^{q_0}(\{(t,x):\frac{T}{2}\leq t\leq T, \delta\leq\phi(t)-|x|\leq2\delta\})}\leq&C\delta_0^{\frac{1}{q_0}}(\delta\phi(T))^{-\frac{1}{q_0}}\parallel G\parallel_{L^{\frac{q_0}{q_0-1}}(\mathbb{R}^{n+1}_+)} \no\\
=&C\delta_0^{\frac{1}{q_0}}\phi(T)^{-\frac{1}{q_0}+\frac{\nu}{2}}\delta^{-\frac{1}{q_0}
-\frac{\nu}{2}}\Big(\frac{\delta}{\phi(T)}\Big)^{\frac{\nu}{2}}
\parallel G\parallel_{L^{\frac{q_0}{q_0-1}}(\mathbb{R}^{n+1}_+)} \no\\
\leq&C\delta_0^{\frac{1}{q_0}}\phi(T)^{-\frac{1}{q_0}+\frac{\nu}{2}}\delta^{-\frac{1}{q_0}-\frac{\nu}{2}}
\parallel G\parallel_{L^{\frac{q_0}{q_0-1}}(\mathbb{R}^{n+1}_+)},
\end{align}
here we have used the fact of  $\delta\lesssim\phi(T)$ due to $2\delta\phi(T_0)\leq\phi(T)$ and $\phi(T_0)\geq1$. Then
\eqref{equ:4.103} together with Lemma \ref{lem:a5} yields estimate \eqref{equ:4.39} in Case (ii).

\subsubsection{The proof of \eqref{equ:4.39} in Case (iii)}\label{sec4:large:iii}

Motivated by the ideas in  $\S 3$ of \cite{Gls}, we shall separate
the related Fourier integral operator in the expression of $v$
into a high frequency part and a low frequency part, subsequently we
handle them with different techniques respectively.
 %At first we notice the facts of $t\geq s$ and $1\leq s\leq2$ (by the scaling argument in \eqref{equ:4.39}).
First, note that the solution \(w\) of \eqref{equ:4.104} can be expressed as
\begin{equation*}
v(t,x)=\int_0^t\{V_2(t, D_x)V_1(s, D_x)-V_1(t, D_x)V_2(s, D_x)\}G(s,x)\md s,
%\label{equ:4.5}
\end{equation*}
where the definitions of \(V_j\) (\(j=1,2\) are given in \eqref{equ:2.3}-\eqref{equ:2.4}.
Then by the analogous analysis in Lemma 3.4 of \cite{HWYin}, we have
\begin{equation}
v(t,x)=\int_0^t\int_{\mathbb{R}^n}e^{i(x\cdot\xi+(\phi(t)-\phi(s))|\xi|)}a(t,s,\xi)\hat{G}(s,\xi)\md\xi \md s,
\label{equ:4.6}
\end{equation}
where the amplitude function $a$ satisfies for $\beta\in \Bbb N_0^n$,
\begin{equation}\label{equ:4.7}
\big| \partial_\xi^\beta a(t,s,\xi)\big|\leq C\big(1+\phi(t)|\xi|\big)^{-\frac{m}{2(m+2)}}\big(1+\phi(s)|\xi|\big)^{-\frac{m}{2(m+2)}}|\xi|^{-\frac{2}{m+2}-|\beta|}.
\end{equation}
To show \eqref{equ:4.39}, note that \(1\leq s\leq2\), we then introduce and deal with such an operator for $\mu\in(0,\f{m}{2(m+2)}]$,
\[v_{\mu}(t,x)=\int_0^t\int_{\Bbb R^n}\int_{\Bbb R^n} e^{i((x-y)\cdot\xi\pm(\phi(t)-\phi(s))|\xi|)}G(s,y)\big(1+\phi(t)|\xi|\big)^{-\frac{m}{2(m+2)}}
\frac{\md\xi}{|\xi|^{\frac{2}{m+2}+\mu}}\md y\md s.\]
Here we point out that the appearance of the factor $\f{1}{|\xi|^{\mu}}$ with $\mu>0$
in $v_{\mu}(t,x)$ will play an important role in
estimating $w(t,x)$ (one can see \eqref{equ:4.105} below).

Set $\tau=\phi(s)-|y|$. Applying H\"{o}lder's inequality, one then has that
\begin{equation}\label{equ:4.55}
\begin{split}
&|v_{\mu}|\le C\delta_0^{\frac{1}{q_0}}\Big(\int_{\delta_0}^{2\delta_0}\big|\int_{\Bbb R^n}\int_{\Bbb R^n} e^{i((x-y)\cdot\xi\pm(\phi(t)-\tau-|y|)|\xi|)}G\big(\phi^{-1}(\tau+|y|),y\big)\big(1+\phi(t)|\xi|\big)^{-\frac{m}{2(m+2)}} \\
&\qquad \times \frac{\md\xi}{|\xi|^{\frac{2}{m+2}+\mu}}\md y
\Big|^{\frac{q_0}{q_0-1}}\md\tau\Big)^{\frac{q_0-1}{q_0}}.
\end{split}
\end{equation}

To treat the integral in \eqref{equ:4.55}, it is enough to consider the phase function with sign minus
since for the case of sign plus the related analysis can be done in the same way. Note that we have
the assertions:
\begin{equation}\label{equ:4.3}
\delta_0\leq\frac{1}{4}\phi(1)
\end{equation}
and
\begin{equation}\label{equ:4.4}
\frac{1}{2}\phi(1)\leq|y|\leq\phi(2).
\end{equation}
Indeed, by $|y|\leq\phi(s)\leq\phi(2)$,
one then has $|y|\geq\phi(s)-2\delta_0\geq\frac{1}{2}\phi(1)$
if $\delta_0\leq\frac{1}{4}\phi(1)$. However, if  $\delta_0>\frac{1}{4}\phi(1)$, then
$10\phi(2)<40\cdot\frac{\phi(2)}{\phi(1)}\delta_0\leq\delta\leq10\phi(2)$ holds,
which yields a contradiction.

Next we start to estimate $\|v_{\mu}\|_{L_x^{q_0}}$. First note that by \eqref{equ:4.30},
\[\phi(t)\geq\phi(\frac{T}{2T_0})=\frac{\phi(T)}{\phi(2T_0)}\cdot\frac{2}{m+2}\geq10\frac{\phi(2)}{\phi(1)}\cdot\phi(1)=10\phi(2),\]
this together with $\tau<\phi(s)<\phi(2)$ yields $\phi(t)\geq\phi(t)-\tau>\frac{1}{2}\phi(t)$.
Thus we can replace $\phi(t)-\tau$ with $\phi(t)$ in \eqref{equ:4.55} and consider
\begin{equation}
\begin{split}
(T_{\mu}g)(t,x)=\int_{\mathbb{R}^n}\int_{\{y\in\mathbb{R}^n:\frac{1}{2}\phi(1)\leq|y|\leq\phi(2)\}}&e^{i[(x-y)\cdot\xi-(\phi(t)-|y|)|\xi|]} \\
&\times \big(1+\phi(t)|\xi|\big)^{-\frac{m}{2(m+2)}}g(y)\frac{\md\xi}{|\xi|^{\frac{2}{m+2}+\mu}}\md y,
\end{split}
\label{equ:H.1}
\end{equation}
where $\phi(t)\geq10\phi(2)-\phi(2)\geq9\phi(2)$ and $\delta<10\phi(2)$, then by Lemma \ref{lem:a2} in Appendix,
\eqref{equ:4.39} follows from
\begin{equation}
\parallel (T_{\mu}g)(t,\cdot)\parallel_{L^{q_0}(\{x:\delta\leq\phi(t)-|x|\leq2\delta\})}\leq C\phi(t)^{\frac{\nu}{2}-\frac{m+4}{q_0(m+2)}}\delta^{-\frac{\nu}{2}
-\frac{1}{q_0}}
\parallel g\parallel_{L^{\frac{q_0}{q_0-1}}(\mathbb{R}^n)}.
\label{equ:4.57}
\end{equation}

Next we focus on the proof of \eqref{equ:4.57}. We shall use the complex interpolation method to establish \eqref{equ:4.57}. To do this, set for \(z\in \mathbb{C}\)
\begin{equation}
\begin{split}
(T_zg)(t,x)=(z-\frac{(m+2)n+2}{2(m+2)})e^{z^2} \int_{\mathbb{R}^n}\int_{\mathbb{R}^n}
&e^{i((x-y)\cdot\xi-(\phi(t)-|y|)|\xi|)} \\
&\times\big(1+\phi(t)|\xi|\big)^{-\frac{m}{2(m+2)}}g(y)\frac{\md\xi}{|\xi|^z}\md y.
\end{split}
\label{equ:4.58}
\end{equation}
For clearer statement on \eqref{equ:4.57}, we shall replace $\frac{\nu}{2}$ by $\frac{\nu}{q_0}$ and rewrite
\eqref{equ:4.57} as
\begin{equation}
\begin{split}
\parallel (T_{\mu}g)(t,\cdot)\parallel_{L^{q_0}(\{x:\delta\leq\phi(t)-|x|\leq2\delta\})}\leq C\phi(t)^{\frac{\nu}{q_0}-\frac{m+4}{q_0(m+2)}}\delta^{-\frac{\nu}{q_0}
-\frac{1}{q_0}}\parallel g\parallel_{L^{\frac{q_0}{q_0-1}}(\mathbb{R}^n)}.
\end{split}
\label{equ:4.56}
\end{equation}
For some suitable scope of $\mu>0$, then \eqref{equ:4.56} would be a consequence of
\begin{equation}
\parallel (T_zg)(t,\cdot)\parallel_{L^\infty(\mathbb{R}^n)}\leq C\phi(t)^{-\frac{n-1}{2}-\frac{m}{2(m+2)}+\mu}\parallel g\parallel_{L^1(\mathbb{R}^n)} \quad\text{with $Rez=\frac{(m+2)n+2}{2(m+2)}+\mu$,}
\label{equ:4.59}
\end{equation}
and
\begin{equation}
\parallel (T_zg)(t,\cdot)\parallel_{L^2(\mathbb{R}^n)}\leq C\phi(t)^{-\frac{m}{2(m+2)}}
(\phi(t)^{\nu}\delta^{-(\nu+1)})^{\frac{1}{m+2}}\parallel g\parallel_{L^2(\mathbb{R}^n)} \quad\text{with $Rez=0$.}
\label{equ:4.60}
\end{equation}
In fact, the interpolation between \eqref{equ:4.60} and \eqref{equ:4.59} gives
\begin{equation*}%\label{equ:4.61}
  \parallel (T_{\mu}g)(t,\cdot)\parallel_{L^{q_0}(\{x:\delta\leq\phi(t)-|x|\leq2\delta\})}\leq C\phi(t)^{\frac{2\nu}{q_0(m+2)}+\frac{4\mu}{(m+2)n+2}-\frac{m+4}{q_0(m+2)}}\delta^{-\frac{2(\nu+1)}{(m+2)q_0}
}\parallel g\parallel_{L^{\frac{q_0}{q_0-1}}(\mathbb{R}^n)}.
\end{equation*}
By $\delta\leq10\phi(2)$, it suffices  to derive \eqref{equ:4.56}
only under the condition of $\frac{2\nu}{q_0(m+2)}+\frac{4\mu}{(m+2)n+2}\leq\frac{\nu}{q_0}$.
Thus, if $0<\mu\leq\min\big(\frac{m}{2(m+2)}, \frac{m((m+2)n+2)}{4(m+2)q_0}\nu\big)$,
then \eqref{equ:4.56} can be proved.

We now prove \eqref{equ:4.59} by the stationary phase method.
To this end, for the cut-off function $\beta$ in \eqref{equ:2.13}, $\lambda_j=2^j$ with $j\geq0$, and
$z=\frac{(m+2)n+2}{2(m+2)}+\mu+i\theta$ with $\theta\in\Bbb R$, we define and estimate the dyadic operator $T_z^jg$ as follows

\begin{equation}
\begin{split}
|T_z^jg|&=\Big| \theta e^{-\theta^2}\int_{\mathbb{R}^n}\int_{\mathbb{R}^n}e^{i[(x-y)\cdot\xi-(\phi(t)-|y|)|\xi|]} \beta(\f{|\xi|}{\lambda_j})\big(1+\phi(t)|\xi|\big)^{-\frac{m}{2(m+2)}}g(y)\frac{\md\xi}{|\xi|^z}\md y \Big| \\
&\leq C\lambda_j^{\frac{n+1}{2}}\big(\phi(t)-\phi(2)\big)^{-\frac{n-1}{2}}\big(\lambda_j\phi(t)\big)^{-\frac{m}{2(m+2)}}
\lambda_j^{-\frac{(m+2)n+2}{2(m+2)}-\mu}\parallel g\parallel_{L^1(\mathbb{R}^n)} \\
&\leq C\lambda_j^{-\mu}\phi(t)^{-\frac{n-1}{2}-\frac{m}{2(m+2)}}
\parallel g\parallel_{L^1(\mathbb{R}^n)}.
\end{split}
\label{equ:4.62}
\end{equation}
Summing up \eqref{equ:4.62} with respect to $j\ge 0$ yields
\begin{equation}\label{equ:4.105}
\bigg\|\sum_{j=0}^\infty T_z^jg\bigg\|_{L^\infty(\mathbb{R}^n)}\leq C
\phi(t)^{-\frac{n-1}{2}-\frac{m}{2(m+2)}}\parallel g\parallel_{L^1(\mathbb{R}^n)}.
\end{equation}
On the other hand, for $|\xi|\leq1$ we have that
\begin{equation}\label{equ:4.64}
\begin{split}
&\Big| \theta e^{-\theta^2}\int_{|\xi|\leq1}\int_{\mathbb{R}^n}e^{i[(x-y)\cdot\xi-(\phi(t)-|y|)|\xi|]} \big(1+\phi(t)|\xi|\big)^{-\frac{m}{2(m+2)}}g(y)\frac{\md\xi}{|\xi|^z}\md y \Big| \\
&\leq C\parallel g\parallel_{L^1(\mathbb{R}^n)} \int_{|\xi|\leq1}\big(1+\phi(t)|\xi|\big)^{-\frac{n-1}{2}-\frac{m}{2(m+2)}}
|\xi|^{-\frac{(m+2)n+2}{2(m+2)}-\mu}\md\xi\\
&\leq C\parallel g\parallel_{L^1(\mathbb{R}^n)} \int_0^1(1+\phi(t)r)^{-\frac{n-1}{2}-\frac{m}{2(m+2)}}
r^{-\frac{n}{2}-\frac{1}{m+2}-\mu}r^{n-1}\md r,
\end{split}
\end{equation}
here we have noted the fact of $n-1-\frac{n}{2}-\frac{1}{m+2}-\mu=\frac{n}{2}-1-\frac{1}{m+2}-\mu>-1$ for $0<\mu\leq\frac{m}{2(m+2)}$, $m\geq1$ and
$n\geq3$, thus the integral in last line of \eqref{equ:4.64} is convergent. In order to give a precise estimate
to \eqref{equ:4.64}, setting $\rho=(1+\phi(t)r)^{\frac{1}{2(m+2)}}$, then

\begin{equation}\label{equ:4.65}
\begin{split}
&\int_0^1(1+\phi(t)r)^{-\frac{n-1}{2}-\frac{m}{2(m+2)}}
r^{-\frac{n}{2}-\frac{1}{m+2}-\mu}r^{n-1}\md r \\
&=\int_1^{(1+\phi(t))^{\frac{1}{2(m+2)}}}\rho^{2-(m+2)n}\Big(\frac{\rho^{2(m+2)}-1}{\phi(t)}\Big)^{\frac{n}{2}-1-\frac{1}{m+2}-\mu}
\md\Big(\frac{\rho^{2(m+2)}-1}{\phi(t)}\Big)\\
&\leq C\phi(t)^{-\frac{n}{2}+\frac{1}{m+2}+\mu}\int_1^{(1+\phi(t))^{\frac{1}{2(m+2)}}}
\frac{\rho^{(m+2)n-2(m+3)-2(m+2)\mu}\rho^{2(m+2)-1}}{\rho^{(m+2)n-2}}\md\rho \\
&=C\phi(t)^{-\frac{n}{2}+\frac{1}{m+2}+\mu}\int_1^{(1+\phi(t))^{\frac{1}{2(m+2)}}}\rho^{-(m+2)\mu-1}\md\rho \\
&\leq C_\mu\phi(t)^{-\frac{n-1}{2}-\frac{m}{2(m+2)}+\mu}.
\end{split}
\end{equation}
Thus combining\eqref{equ:4.64} and \eqref{equ:4.65} yields \eqref{equ:4.59} with a positive constant $C$ depends
only on $m$, $n$ and $\mu$.

To get \eqref{equ:4.60}, the small frequencies and large frequencies will be treated separately
in the Fourier integral operator of \eqref{equ:4.58}.
As in \cite{Gls}, we shall use Sobolev trace theorem to handle the small frequency part.
More specifically, we first introduce a function $\rho\in C^\infty(\mathbb{R}^n)$ such that
\begin{equation*}
\rho(\xi)=
\left\{ \enspace
\begin{aligned}
1, \quad &|\xi|\geq2, \\
0, \quad &|\xi|\leq1.
\end{aligned}
\right.
\end{equation*}
For $\alpha=1+\nu$, let
$$(T_zg)(t,x)=(R_zg)(t,x)+(S_zg)(t,x),$$
where
\begin{align}\label{equ:4.115}
(R_zg)(t,x)=&\Big(z-\frac{(m+2)n+2}{2(m+2)}\Big)e^{z^2}
\int_{\mathbb{R}^n}\int_{\frac{1}{2}\phi(1)\leq|y|\leq\phi(2)}e^{i((x-y)\cdot\xi-(\phi(t)-|y|)|\xi|)}\no\\
&\qquad\qquad \qquad \qquad\times\big(1+\phi(t)|\xi|\big)^{-\frac{m}{2(m+2)}}\Big(1-\rho\big(\phi(t)^{1-\alpha}\delta^\alpha\xi\big)\Big)
g(y)\frac{\md\xi}{|\xi|^z}\md y,\no\\
(S_zg)(t,x)=&\Big(z-\frac{(m+2)n+2}{2(m+2)}\Big)e^{z^2}
\int_{\mathbb{R}^n}\int_{\frac{1}{2}\phi(1)\leq|y|\leq\phi(2)}e^{i((x-y)\cdot\xi-(\phi(t)-|y|)|\xi|)} \no\\
&\qquad\qquad \qquad \qquad\times\big(1+\phi(t)|\xi|\big)^{-\frac{m}{2(m+2)}}\rho\big(\phi(t)^{1-\alpha}\delta^\alpha\xi\big)g(y)\frac{\md\xi}{|\xi|^z}\md y.
\end{align}
It follows from Lemma \ref{lem:a3} and Lemma \ref{lem:a4} in Appendix that

\begin{equation}
\parallel (R_zg)(t,\cdot)\parallel_{L^2(\mathbb{R}^n)}\leq C\phi(t)^{-\frac{m}{2(m+2)}}(\phi(t)^{\alpha-1}\delta^{-\alpha})^{\frac{1}{m+2}}\parallel g\parallel_{L^2(\mathbb{R}^n)}
\quad \text{for $Rez=0$},
\label{equ:4.66}
\end{equation}
and
\begin{equation}
\begin{split}
\parallel S_zg(t,\cdot)&\parallel_{L^2(\{x:\delta\leq\phi(t)-|x|\leq2\delta\})}\leq C\delta^{-\frac{1}{2}}(\phi(t)^\alpha\delta^{-\alpha})^{-\frac{m}{2(m+2)}}\parallel g\parallel_{L^2(\mathbb{R}^n)}\\
=&C\phi(t)^{-\frac{m}{2(m+2)}}(\phi(t)^{\alpha-1}\delta^{-\alpha})^{\frac{1}{m+2}}\Big(\frac{\delta}{\phi(t)}\Big)^{\frac{\alpha-1}{2}}
\parallel g\parallel_{L^2(\mathbb{R}^n)} \qquad \text{for $Rez=0$}.
\end{split}
\label{equ:4.67}
\end{equation}
Note that \eqref{equ:4.67} is actually stronger than \eqref{equ:4.60} and \eqref{equ:4.66} by our assumptions
of $\delta\leq\delta\phi(T_0)\leq\phi(T)$ and $\frac{T}{2}\leq t\leq T$. \eqref{equ:4.66} together with \eqref{equ:4.67}
yields \eqref{equ:4.60}.

Collecting all the analysis above in Case (i)- Case (iii), \eqref{equ:4.39} is proved
in the case of large time part.

\subsection{Estimate for small time}
Now we turn to the estimate of \(w^1\), which corresponds to the small time
part of $w$ in \eqref{equ:4.1}.
Similar to the treatment on \(w^0\) in the large time part of $w$, we shall prove
\eqref{equ:4.39}.

As in the Subsection 4.1, we will divide the proof of \eqref{equ:4.39} into the following two parts according to
the different values of $\frac{\delta}{\delta_0}$:

(i) $\delta_0\leq\delta\leq40\cdot\frac{\phi(2)}{\phi(1)}\delta_0$;

(ii) $40\cdot\frac{\phi(2)}{\phi(1)}\delta_0\leq\delta\leq\phi(T)$.

\subsubsection{The proof of \eqref{equ:4.39} in Case (i)}

In this case, one has that $\phi(T)>\phi(T_0)\geq1$ and $\delta\phi(T_0)\leq\phi(T)$. To prove \eqref{equ:4.39}, it suffices to show
\begin{equation}\label{D1}
\phi(T)^{\frac{1}{q_0}}\parallel v\parallel_{L^{q_0}(\{(t,x):\frac{T}{2}\leq t\leq T, \delta\leq\phi(t)-|x|\leq2\delta\})}\leq C\parallel G\parallel_{L^{\frac{q_0}{q_0-1}}(\mathbb{R}^{1+n}_+)}.
%\label{equ:4.68}
\end{equation}
By $\phi(1)\leq\phi(T)\leq10\phi(2)$, we only need to prove
\begin{equation}\label{D2}
\parallel v\parallel_{L^{q_0}(\{(t,x):\frac{T}{2}\leq t\leq T, \delta\leq\phi(t)-|x|\leq2\delta\})}\leq C\parallel G\parallel_{L^{\frac{q_0}{q_0-1}}(\mathbb{R}^{1+n}_+)}.
\end{equation}
But this just only follows from the un-weighted Strichartz estimate in  \cite{HWYin}
(see Lemma 3.4 of \cite{HWYin}).

\subsubsection{The proof of \eqref{equ:4.39} in Case (ii)}

In this case, we only need to prove
\begin{equation}
\delta^{\frac{1}{q_0}+\frac{\nu}{2}}\parallel v\parallel_{L^{q_0}(\{(t,x):\frac{T}{2}\leq t\leq T, \delta\leq\phi(t)-|x|\leq2\delta\})}\leq C\phi(T)^{\frac{\nu}{2}}\delta_0^{\frac{1}{q_0}}\parallel G\parallel_{L^{\frac{q_0}{q_0-1}}(\mathbb{R}^{1+n}_+)}.
\label{equ:4.68}
\end{equation}
By \eqref{equ:4.6}-\eqref{equ:4.7}, we can write
\begin{equation}\label{D4}
\begin{split}
v=&\int_0^t\int_{\Bbb R^n}\int_{\Bbb R^n} e^{i((x-y)\cdot\xi\pm(\phi(t)-\phi(s))|\xi|)}G(s,y) \\ &\times\big(1+\phi(t)|\xi|\big)^{-\frac{m}{2(m+2)}}\big(1+\phi(s)|\xi|\big)^{-\frac{m}{2(m+2)}}\frac{\md\xi}{|\xi|^{\frac{2}{m+2}}}\md y\md s.
\end{split}
\end{equation}
Set \(\tau=\phi(s)-|y|\). Then we have
\begin{equation}\label{equ:4.70}
\begin{split}
|v|\le &C\delta_0^{\frac{1}{q_0}}\Big(\int_{\delta_0}^{2\delta_0}\Big|\int_{\Bbb R^n}\int_{\Bbb R^n} e^{i((x-y)\cdot\xi\pm(\phi(t)-\tau-|y|)|\xi|)}G(\phi^{-1}\big(\tau+|y|\big),y) \\
&\times\big(1+\phi(t)|\xi|\big)^{-\frac{m}{2(m+2)}}\big(1+\phi(s)|\xi|\big)^{-\frac{m}{2(m+2)}}\frac{\md\xi}{|\xi|^{\frac{2}{m+2}}}\md y
\Big|^{\frac{q_0}{q_0-1}}d\tau\Big)^{\frac{q_0-1}{q_0}}.
\end{split}
\end{equation}
As in \eqref{equ:4.3} and \eqref{equ:4.4}, one has \(\delta_0\leq\frac{1}{4}\phi(1)\)
and \(\frac{1}{2}\phi(1)\leq|y|\leq\phi(2)\). Then
\[\tau=\phi(s)-|y|\leq\phi(s)-\frac{1}{2}\phi(1).\]
This, together with \(t\geq s\), yields
\[\f{\phi(t)-\tau}{\phi(t)}\geq\frac{\phi(t)-\phi(s)+\phi(1)/2}{\phi(t)}\geq\frac{\phi(1)}{2\phi(t)}\geq\frac{\phi(1)}{20\phi(2)}>0.\]
In addition, \(\phi(t)-\tau\leq\phi(t)\) holds. Thus, we can replace \(\phi(t)-\tau\) with \(\phi(t)\) in \eqref{equ:4.70} and consider
an operator as follows
\begin{equation*}
\begin{split}
(\t Tg)(t,x)=\int_{\mathbb{R}^n}\int_{D_y}e^{i((x-y)\cdot\xi-(\phi(t)-|y|)|\xi|)} \big(1+\phi(t)|\xi|\big)^{-\frac{m}{2(m+2)}}\big(1+\phi(s)|\xi|\big)^{-\frac{m}{2(m+2)}}g(y)\frac{\md\xi}{|\xi|^{\frac{2}{m+2}}}\md y,
\end{split}
%\label{equ:4.71}
\end{equation*}
where \(D_y=\{y\in \mathbb{R}^n: \frac{1}{2}\phi(1)\leq|y|\leq\phi(2)\}\).
By a computation similar to Lemma A.2, we know that \eqref{equ:4.39} follows from
\begin{equation*}
\parallel (\tilde{T}g)(t,\cdot)\parallel_{L^{q_0}(\{x:\delta\leq\phi(t)-|x|\leq2\delta\})}\leq C\phi(t)^{\frac{\nu}{2}}\delta^{-\frac{\nu}{2}
-\frac{1}{q_0}}
\parallel g\parallel_{L^{\frac{q_0}{q_0-1}}(\mathbb{R}^n)}.
%\label{equ:4.72}
\end{equation*}
As in \eqref{equ:4.56}, we shall replace \(\f{\nu}{2}\) with \(\f{\nu}{q_0}\) and prove
\begin{equation}
\parallel (\tilde{T}g)(t,\cdot)\parallel_{L^{q_0}(\{x:\delta\leq\phi(t)-|x|\leq2\delta\})}\leq C\phi(t)^{\frac{\nu}{q_0}}\delta^{-\frac{\nu}{q_0}
-\frac{1}{q_0}}
\parallel g\parallel_{L^{\frac{q_0}{q_0-1}}(\mathbb{R}^n)}.
\label{equ:4.72}
\end{equation}
To deal with \(\tilde{T}g\), we write
\begin{equation}\label{D3}
\begin{split}
&\tilde{T}g \\
&=\int_{|\xi|\leq1}\int_{D_y}e^{i((x-y)\cdot\xi-(\phi(t)-|y|)|\xi|)}
\big(1+\phi(t)|\xi|\big)^{-\frac{m}{2(m+2)}}\big(1+\phi(s)|\xi|\big)^{-\frac{m}{2(m+2)}}\frac{\md\xi}{|\xi|^{\frac{2}{m+2}}}g(y)\md y\\
&\quad +\int_{|\xi|>1}\int_{D_y}e^{i((x-y)\cdot\xi-(\phi(t)-|y|)|\xi|)} \big(1+\phi(t)|\xi|\big)^{-\frac{m}{2(m+2)}}\big(1+\phi(s)|\xi|\big)^{-\frac{m}{2(m+2)}}\frac{\md\xi}{|\xi|^{\frac{2}{m+2}}}g(y)\md y\\
& =:\t T^1g+\tilde{T}^2g.
\end{split}
\end{equation}
By the uniform boundedness of $\phi(t)$ and $\phi(s)$ from below and above, we may consider
\begin{equation*}
  T^2g =\int_{|\xi|>1}\int_{\{y\in\mathbb{R}^n:\frac{1}{2}\phi(1)\leq|y|\leq\phi(2)\}}e^{i((x-y)\cdot\xi-(\phi(t)-|y|)|\xi|)} g(y)\frac{\md\xi}{|\xi|}\md y
\end{equation*}
instead of $\tilde{T}^2g$ in \eqref{D3}.
Next we prove \eqref{equ:4.72} for \(\t T^1 g\) and \(T^2 g\). As in the proof of previous ``large time" part, we shall use the
complex interpolation method to establish \eqref{equ:4.72}. More specifically, we set that for $z\in\Bbb C$,
\begin{align*}
T^1_zg &=\int_{|\xi|\leq1}\int_{\{y\in\mathbb{R}^n:\frac{1}{2}\phi(1)\leq|y|\leq\phi(2)\}}e^{i[(x-y)\cdot\xi-(\phi(t)-|y|)|\xi|]} \\
&\qquad \times \big(1+\phi(t)|\xi|\big)^{-\frac{m}{2(m+2)}}\big(1+\phi(s)|\xi|\big)^{-\frac{m}{2(m+2)}}\frac{\md\xi}{|\xi|^{z}}g(y)\md y
\end{align*}
and
\begin{equation*}
  T^2_z g =\int_{|\xi|>1}\int_{\{y\in\mathbb{R}^n:\frac{1}{2}\phi(1)\leq|y|\leq\phi(2)\}}e^{i[(x-y)\cdot\xi-(\phi(t)-|y|)|\xi|]} g(y)\frac{\md\xi}{|\xi|^z}\md y.
\end{equation*}
We now prove the following inequalities
\begin{equation}\label{equ:4.74}
  \|T^1_zg(t,\cdot)\|_{L^\infty(\R^n)}\leq\|g\|_{L^1(\R^n)}, \quad \operatorname{Re}z=\f{(m+2)n+2}{2(m+2)};
\end{equation}
\begin{equation}\label{equ:4.75}
  \|T^1_zg(t,\cdot)\|_{L^2(\R^n)}\leq(\phi(t)^{\nu}\delta^{-(\nu+1)})^{\frac{1}{m+2}}\|g\|_{L^2(\R^n)}, \quad \operatorname{Re}z=0;
\end{equation}
\begin{equation}\label{equ:4.76}
  \|T^2_zg(t,\cdot)\|_{L^\infty(\R^n)}\leq\|g\|_{L^1(\R^n)}, \quad \operatorname{Re}z=\f{(m+2)n+2}{4};
\end{equation}
\begin{equation}\label{equ:4.77}
  \|T^2_zg(t,\cdot)\|_{L^2(\R^n)}\leq(\phi(t)^{\nu}\delta^{-(\nu+1)})^{\frac{1}{m+2}}\|g\|_{L^2(\R^n)}, \quad \operatorname{Re}z=0.
\end{equation}
From \eqref{equ:4.74}-\eqref{equ:4.77}, by the interpolation among the \(L^1-L^\infty\) estimates  (\eqref{equ:4.74} and \eqref{equ:4.76})
and the \(L^2-L^2\) estimates (\eqref{equ:4.75} and \eqref{equ:4.77}), we get that for $2\le q_0\le \infty$ and \(j=1,2\),
\begin{equation}\label{equ:4.113}
\begin{split}
\|T^j_zg(t,\cdot)\|_{L^{q_0}(\R^n)} & \leq C\big(\phi(t)^{\nu}\delta^{-(\nu+1)}\big)^{\frac{2}{(m+2)q_0}}\|g\|_{L^{\frac{q_0}{q_0-1}}(\R^n)} \\
&=C\phi(t)^{-\frac{\nu}{q_0}\frac{m}{m+2}}
\delta^{\frac{\nu+1}{q_0}\frac{m}{m+2}}\phi(t)^{\frac{\nu}{q_0}}\delta^{-\frac{\nu+1}{q_0}}\|g\|_{L^{\frac{q_0}{q_0-1}}(\R^n)} \\
&\leq C\phi(t)^{\frac{\nu}{q_0}}\delta^{-\frac{\nu+1}{q_0}}\|g\|_{L^{\frac{q_0}{q_0-1}}(\R^n)},
\end{split}
\end{equation}
here we have used the facts of  \(\phi(t)\geq\phi(1)\) and \(\delta\leq10\phi(2)\)
in the last inequality. From \eqref {equ:4.113}, \eqref{equ:4.72} is then proved.

Next we start to establish \eqref{equ:4.74}-\eqref{equ:4.77}.

\paragraph{\(\mathbf{L^1-L^\infty}\) estimates:}

At first, we derive the \(L^1-L^\infty\) estimate for \(T^1_zg\). It follows from direct computation that
\begin{equation*}
\begin{split}
|T^1_zg| & \leq\bigg|\int_{|\xi|\leq1}|\xi|^{-\frac{(m+2)n+2}{2(m+2)}}\md\xi\bigg|\|g\|_{L^1(\R^n)} \\
& \leq C\bigg|\int_{r\leq1}r^{-\f{n}{2}-\f{1}{m+2}}r^{n-1}\md r\bigg|\|g\|_{L^1(\R^n)} \\
& \leq C\|g\|_{L^1(\R^n)}, \quad \operatorname{Re}z=\f{(m+2)n+2}{2(m+2)}.
\end{split}
\end{equation*}
Thus \eqref{equ:4.74} is established. To study \(T^2_zg\), we shall introduce and estimate the corresponding
dyadic operators as in \eqref{equ:4.79}. Let \(\operatorname{Re}z=\f{(m+2)n+2}{4}\)
and \(\operatorname{Im}z=\theta\), \(\lambda_j=2^j\), then for \(j\geq0\),
\begin{equation*}
\begin{split}
|T^{2,j}_zg| & =\bigg|\theta e^{-\theta^2}\int_{\R^n}\int_{|\xi|\geq1}e^{i[(x-y)\cdot\xi-(\phi(t)-|y|)|\xi|]}
\beta\big(\f{|\xi|}{\lambda_j}\big)g(y)
\f{\md\xi}{|\xi|^z}\md y\bigg| \\
&\leq C\lambda_j^n \lambda_j^{-\f{(m+2)n+2}{4}}\|g\|_{L^1(\R^n)}.
\end{split}
\end{equation*}
If \(m\geq2\), then
\[|T^{2,j}_zg|\leq C\lambda_j^{-\f{1}{2}}\|g\|_{L^1(\R^n)}.\]
Summing up $T^{2,j}_zg$ with respect to \(j\) yields \eqref{equ:4.76}.

For the case of \(m=1\), we require a more careful analysis on $T^{2,j}_zg$
by applying the stationary phase method. For this purpose,
choosing a \(\kappa>0\) such that \(-\f{n}{4}+\f{1}{2}-\kappa>-\f{n-1}{2}\), then it follows
from the stationary phase method that for $\operatorname{Re}z=\f{3n+2}{4}$,
\begin{equation}\label{equ:4.116}
  \begin{split}
     \bigg|\int_{|\xi|\geq1}&e^{i[(x-y)\cdot\xi-(\phi(t)-|y|)|\xi|]}\beta\big(\f{|\xi|}{\lambda_j}\big)\f{\md\xi}{|\xi|^z}\bigg| \\
       &\leq C\lambda_j^n \lambda_j^{-\f{3n+2}{4}}\Big(1+\big(\phi(t)-|y|\big)\lambda_j\Big)^{-\f{n}{4}+\f{1}{2}-\kappa}.
  \end{split}
\end{equation}
Note that as in \eqref{equ:H.1}, the factor \(\phi(t)\) in \eqref{equ:4.116} is
actually \(\phi(t)-\tau\), then
\[\phi(t)-|y|=\phi(t)-\tau-|y|=\phi(t)-\phi(s).\]
Recall that by our assumption, \(\operatorname{supp}v\subseteq\{(t,x): \delta\leq\phi(t)-|x|\leq2\delta,T/2\leq t\leq T\}\) and
\(\operatorname{supp}G\subseteq\{(s,y): \delta_0\leq\phi(s)-|y|\leq2\delta_0,1\leq s\leq 2\}\).
Together with \(40\cdot\frac{\phi(2)}{\phi(1)}\delta_0\leq\delta\), this yields \(\phi(t)-\phi(s)\geq C\delta\). Then
\begin{equation}\label{equ:4.107}
\begin{split}
\bigg|\int_{\R^n}\int_{|\xi|\geq1}&e^{i[(x-y)\cdot\xi-(\phi(t)-|y|)|\xi|]}
\beta\big(\f{|\xi|}{\lambda_j}\big)g(y)
\f{\md\xi}{|\xi|^z}\md y\bigg|\\
&\leq C\lambda_j^n \lambda_j^{-\f{3n+2}{4}}\big(\delta\lambda_j\big)^{-\f{n}{4}+\f{1}{2}-\kappa}\|g\|_{L^1(\R^n)}\\
&\leq C\lambda_j^{-\kappa}\delta^{-\f{n}{4}+\f{1}{2}-\kappa}\|g\|_{L^1(\R^n)}.
\end{split}
\end{equation}
Subsequently,
\begin{equation}\label{equ:4.108}
  |T^{2,j}_zg|\leq C\lambda_j^{-\kappa}\delta^{-\f{n}{4}+\f{1}{2}-\kappa}\|g\|_{L^1(\R^n)} \qquad
  \text{with $\operatorname{Re}z=\f{3n+2}{4}$}.
\end{equation}
Summering up all the estimates on $j\ge 0$ in \eqref{equ:4.108}, we get
\begin{equation}\label{equ:4.109}
|T^{2}_zg|\leq C\delta^{-\f{n}{4}+\f{1}{2}-\kappa}\|g\|_{L^1(\R^n)} \qquad \text{with
$\operatorname{Re}z=\f{3n+2}{4}$.}
\end{equation}
In order to apply \eqref{equ:4.109} to establish \eqref{equ:4.72}, we shall need a
new version of \eqref{equ:4.77}. More specifically, for certain \(\gamma\in\R\), we intend to show
\begin{equation}\label{equ:4.110}
\|T^2_zg(t,\cdot)\|_{L^2(\R^n)}\leq(\phi(t)^{\nu}\delta^{-\gamma})^{\frac{1}{3}}\|g\|_{L^2(\R^n)} \qquad
\text{with $\operatorname{Re}z=0$.}
\end{equation}
If this is done, by interpolating \eqref{equ:4.110} with \eqref{equ:4.109}, we then have
\begin{equation}\label{equ:4.111}
\begin{split}
\|T^j_zg(t,\cdot)\|_{L^{q_0}(\R^n)} & \le C\phi(t)^{-\frac{\nu}{3q_0}}
\phi(t)^{\frac{\nu}{q_0}}\delta^{-\frac{\gamma}{3}\frac{2}{q_0}+\big(
-\f{n}{4}+\f{1}{2}-\kappa\big)\big(1-\frac{2}{q_0}\big)}\|g\|_{L^{\frac{q_0}{q_0-1}}(\R^n)}.
\end{split}
\end{equation}
Note that \(\delta\leq10\phi(2)\). Then in order to get \eqref{equ:4.72} by \eqref{equ:4.111}, we only need
\begin{equation}\label{equ:4.114}
-\frac{\gamma}{3}\frac{2}{q_0}+\big(-\f{n}{4}+\f{1}{2}-\kappa\big)\big(1-\frac{2}{q_0}\big)\geq -(\nu+1)\frac{1}{q_0},
\end{equation}
which means
\[\gamma<\f{1}{2}+\f{3\nu}{2}+\frac{4-12\kappa}{3n-2}.\]
This can be easily achieved. Indeed, if \(n=3\), one can let \(\kappa\) be sufficiently small, then \(\gamma>1\)
is chosen; if \(n\geq4\), we can choose \(\gamma>\f{1}{2}\).

To prove \eqref{equ:4.110} for suitably  chosen \(\gamma\) above, we set \(\alpha=1+\nu\) and replace the cut-off function \(\rho\big(\phi(t)^{1-\alpha}\delta^\alpha\xi\big)\) in \eqref{equ:4.115}
with \(\rho\big(\phi(t)^{1-\alpha}\delta^\gamma\xi\big)\). Then \eqref{equ:4.110}
can be obtained by repeating the arguments of \eqref{equ:4.66} and \eqref{equ:4.67} with slight modifications,
here we omit the details.

\paragraph{\(\mathbf{L^2-L^2}\) estimates:}

Since \eqref{equ:4.77} has been replaced by \eqref{equ:4.110} for \(m=1\), the left
things are to prove \eqref{equ:4.75} for all \(m\geq1\) and \eqref{equ:4.77} for \(m\geq2\).
For this purpose, one just notes that \(|x|\leq\phi(t)\leq10\phi(2)\)
holds, thereafter the integral domain of \(v\) in \eqref{D4} is bounded, thus it only suffices to repeat the proofs in
Subsection \ref{sec4:large:iii} to get \eqref{equ:4.75} and \eqref{equ:4.77}, and the related details can be omitted here.

\section{The proof of (3.4)}

Suppose that $w$ solves \eqref{equ:3.1}, where $F\equiv0$ if $\phi(t)-|x|<1$. By Theorem 2.1 of \cite{Yag2}, we have
\[\parallel w(t,\cdot)\parallel_{L^2(\mathbb{R}^n)}\leq Ct\int_0^t\parallel F(s,\cdot)\parallel_{L^2(\mathbb{R}^n)}\md s,\]
which yields that for $0\leq t\leq5$,
\[\parallel w\parallel_{L^2([0,5]\times\mathbb{R}^n)}\leq C\parallel F\parallel_{L^2([0,5]\times\mathbb{R}^n)}.\]
Note that $\phi(t)-|x|$ is bounded from below and above for $0\leq t\leq5$, thus for any $\nu>0$,

\begin{equation}
\begin{split}
\Big\|\big(\phi(t)^2-|x|^2\big)^{\frac{m-2}{2(m+2)}-\nu}w\Big\|_{L^2([0,5]\times\mathbb{R}^n)}\leq C \Big\|\big(\phi(t)^2-|x|^2\big)^{\frac{1}{2}+\nu}F\Big\|_{L^2(\mathbb{R}^{n+1}_+)}.
\end{split}
\label{E1}
\end{equation}
Next we suppose $T\geq10$. We split $w$ as $w=w^0+w^1$, where for $j=0,1,$
\begin{equation*}
\left\{ \enspace
\begin{aligned}
&(\partial_t^2-t^m \Delta)w^j =F^j,   \\
&w^j(0,x)=0, \quad \partial_{t} w^j(0,x)=0
\end{aligned}
\right.
\end{equation*}
with
\begin{equation*}
F^0=
\left\{ \enspace
\begin{aligned}
&F,  &&\phi(t)\leq\frac{\phi(1)\phi(T)}{10\phi(2)}, \\
&0, &&\phi(t)>\frac{\phi(T)\phi(1)}{10\phi(2)}\\
\end{aligned}
\right.
\end{equation*}
and $F=F^0+F^1$. Then in order to prove (3.4), it suffices to show that for $j=0$, $1$,

\begin{equation}
\begin{split}
\Big\|\big(\phi(t)^2-|x|^2\big)^{\frac{m-2}{2(m+2)}-\nu}&w^j\Big\|_{L^2(\{(t,x):\frac{T}{2}\leq t\leq T\})}\leq C
\phi(T)^{-\frac{\nu}{4}}\Big\|\big(\phi(t)^2-|x|^2\big)^{\frac{1}{2}+\nu}F^j\Big\|_{L^2(\mathbb{R}^{n+1}_+)}.
\end{split}
\label{equ:5.2}
\end{equation}
Note that by the analogous treatment on $w^j$ as in \eqref{equ:4.34}-\eqref{equ:4.39},
\eqref{equ:5.2} would follow from
\begin{equation}
\phi(T)^{\frac{m-2}{2(m+2)}-\frac{\nu}{2}}\delta^{\frac{m-2}{2(m+2)}+\frac{\nu}{2}}\parallel v\parallel_{L^2(\{(t,x):\frac{T}{2}\leq t\leq T, \delta\leq\phi(t)-|x|\leq2\delta\})}\leq C\delta_0^{\frac{1}{2}}\parallel G\parallel_{L^2(\mathbb{R}^{n+1}_+)},
\label{equ:5.7}
\end{equation}
where $\operatorname{supp}G\subseteq\{(t,x): 1\leq t\leq2, \delta_0\leq\phi(t)-|x|\leq2\delta_0\}$, and $\delta\geq\delta_0$.

\subsection{Estimate of \(w^0\)}
Note that $\phi(T)\geq10\phi(2)$
holds for $(t,x)\in \operatorname{supp} w^0$. As in Section \ref{sec4},
we shall deal with the estimates in different cases according to the different scales of \(\delta\).

\subsubsection{$\mathbf{\delta\geq10\phi(2)}$}

As in Subsection \ref{sec4:large:i}, we shall use the pointwise estimate to handle the case of $\phi(t)-|x|\geq\delta\geq10\phi(2)$.
We now write
\[v=\sum_{j=-\infty}^{\infty}v_j=\sum_{j=-\infty}^{\infty}\int_{\mathbb{R}^n}\int_{\mathbb{R}^n} K_j(t,x;s,y)G(s,y)dyds,\]
where
\begin{equation*}
 K_j(t,x;s,y)=\int_{\mathbb{R}^n}e^{i((x-y)\cdot\xi+(\phi(t)-\phi(s))|\xi|)}
\beta\Big(\frac{|\xi|}{2^j}\Big)a(t,s,\xi)\hat{G}(s,\xi) \md\xi.
%\label{equ:3.10}
\end{equation*}
By \eqref{equ:4.51} and H\"{o}lder's inequality, we arrive at
\[|v_j|\leq\parallel K_j(t,x;s,y)(\phi^2(t)-|x|^2)^{\frac{1}{2}}\parallel_{L^2_{s,y}}\parallel(\phi^2(t)-|x|^2)^{-\frac{1}{2}}G(s,y)\parallel_{L^2_{s,y}}.
\]
Taking $N=\f{n}{2}-\f{1}{m+2}$ in \eqref{equ:4.51} and repeating the computations of \eqref{equ:4.52} and \eqref{equ:4.53}, we have
\begin{equation*}
\begin{split}
&\parallel K_j(t,x;s,y)(\phi^2(t)-|x|^2)^{\frac{1}{2}}\parallel_{L^2_{s,y}} \\
&\leq C\delta_0^{\frac{1}{2}}\phi(t)^{-\frac{n-1}{2}+\frac{1}{2}-\frac{m}{2(m+2)}}(\phi(t)-|x|)^{\frac{1}{m+2}-\frac{n}{2}+\frac{1}{2}}
\end{split}
\end{equation*}
and
\[\big(\int_{\mathbb{R}^n}\int_{\mathbb{R}^n}\big\{(\phi(t)^2-|x|^2)^{-\frac{1}{2}}G(s,y)\big\}^2\md y\md s\big)^{\frac{1}{2}}
\leq C(\delta\phi(T))^{-\frac{1}{2}}\parallel G\parallel_{L^2(\mathbb{R}^{n+1}_+)}.\]
In addition, by $-\frac{n-1}{2}+\frac{1}{m+2}<-\frac{1}{2}$ for $n\geq3$ and $m\geq1$, a direct computation yields
\begin{equation*}
\begin{split}
\Big\|&\phi(t)^{-\frac{n-1}{2}+\frac{1}{2}-\frac{m}{2(m+2)}}\big(\phi(t)-|x|\big)^{-\frac{n-1}{2}+\frac{1}{m+2}}\Big\|_{L_{t,x}^2} \\
&\leq C_n\big(\int_{\frac{T}{2}}^T\phi(t)^{-(n-1)+\frac{2}{m+2}}\int_0^{\phi(t)-10\phi(2)}
(\phi(t)-r)^{\frac{2}{m+2}-(n-1)}r^{n-1}\md r\md t\big)^{\frac{1}{2}} \\
&\leq C\big(\int_{\frac{T}{2}}^T\phi(t)^{\frac{2}{m+2}}\md t\big)^{\frac{1}{2}} \\
&\leq CT.
\end{split}
%\label{equ:5.8}
\end{equation*}
Thus we obtain
\begin{equation}
\begin{split}
&\phi(T)^{\frac{m-2}{2(m+2)}-\frac{\nu}{2}}\delta^{\frac{m-2}{2(m+2)}+\frac{\nu}{2}}\parallel v\parallel_{L^2(\{(t,x):\frac{T}{2}\leq t\leq T, \delta\leq\phi(t)-|x|\leq2\delta\})} \\
&\leq\phi(T)^{\frac{m-2}{2(m+2)}-\frac{\nu}{2}}\delta^{\frac{1}{2}+\frac{\nu}{2}}\delta_0^{\frac{1}{2}}(\delta\phi(T))^{-\frac{1}{2}}
T\parallel G\parallel_{L^2(\mathbb{R}^{n+1}_+)} \\
&\leq C\delta_0^{\frac{1}{2}}\parallel G\parallel_{L^2(\mathbb{R}^{n+1}_+)}
\end{split}
\label{E2}
\end{equation}
and then \eqref{equ:5.7} is proved.

\subsubsection{$\mathbf{\delta_0\leq\delta\leq10\phi(2)}$}\label{sec4:w0:small}

Next we study \eqref{equ:5.7} for the case of $\phi(t)-|x|\leq10\phi(2)$.
At first, we claim that under certain restrictions on the variable $\xi$,
this situation can be treated as in the proof of \eqref{equ:4.60} in $\S 4$.
Indeed, recalling \eqref{equ:4.6}-\eqref{equ:4.7}, and noting $t\geq s\geq1$, then we can assume
\[v=\int_0^t\int_{\mathbb{R}^n}\int_{\mathbb{R}^n}e^{i((x-y)\cdot\xi+(\phi(t)-\phi(s))|\xi|)}\phi(t)^{-\frac{m}{2(m+2)}}
|\xi|^{-1}G(s,y)\md y\md\xi \md s.\]
As in the \eqref{sec4:large:iii} of $\S 4$ for the proof of (3.3), we again split $v$ into a low
frequency part and a high frequency part respectively.
To this end, we choose a function $\beta\in C_0^\infty(\mathbb{R}^n)$ satisfying $\beta=1$ near the origin
such that $v=v_0+v_1$, where
\begin{equation*}
\begin{split}
v_1&=\int_0^t\int_{\mathbb{R}^n}\int_{\mathbb{R}^n}e^{i((x-y)\cdot\xi+(\phi(t)-\phi(s))|\xi|)}
\phi(t)^{-\frac{m}{2(m+2)}}\frac{1-\beta(\delta\xi)}{|\xi|}G(s,y)\md y\md\xi \md s.\\
%&=\int_1^2T_1(t,s,x)\md s.
\end{split}
\end{equation*}
If we set \(\phi(s)=|y|+\tau\) and use H\"{o}lder's inequality as in \eqref{equ:4.55}, then
\begin{equation*}
\begin{split}
|v_1|&\leq C\delta_0^{\f{1}{2}}\bigg(\int_{\delta_0}^{2\delta_0}\bigg|\int_{\mathbb{R}^n}\int_{\mathbb{R}^n}e^{i((x-y)\cdot\xi+(\phi(t)-\tau-|y|)|\xi|)}
\phi(t)^{-\frac{m}{2(m+2)}}\frac{1-\beta(\delta\xi)}{|\xi|}G(s,y)\md y\md\xi\bigg|^2
\md \tau\bigg)^{\f{1}{2}}\\
&=:C\delta_0^{\f{1}{2}}\bigg(\int_{\delta_0}^{2\delta_0}|T_1(t,\tau,\cdot)|^2\md \tau\bigg)^{\f{1}{2}}.
\end{split}
\end{equation*}
Note $\frac{1-\beta(\delta\xi)}{|\xi|}=O(\delta)$. Then the expression of $v_1$ is similar to \eqref{equ:4.58} with $Rez=0$
%except the extra term $\delta\phi(t)^{-\frac{m}{2(m+2)}}$
. Consequently we can apply the method of \eqref{equ:4.60} to get
\[\parallel T_1(t,\tau,\cdot)\parallel_{L^2(\mathbb{R}^n)}\leq C\big(\phi(t)-\tau\big)^{\frac{\nu}{2}-\frac{m}{2(m+2)}}
\delta^{-\frac{\nu+1}{2}+1}\parallel G(s,\cdot)\parallel_{L^2(\mathbb{R}^n)},\]
which derives
\[\parallel v_1\parallel_{L^2}\leq C\delta_0^{\frac{1}{2}}\delta^{-\frac{\nu+1}{2}+1}\phi(T)^{\frac{\nu}{2}-\frac{m-2}{2(m+2)}}\parallel G\parallel_{L^2(\mathbb{R}^{n+1}_+)}.\]
Due to $\delta\leq10\phi(2)$, the estimate \eqref{equ:5.7} for $v_1$ follows immediately.

We now estimate $v_0$. At first, one notes that
\begin{equation}\label{equ:5.15}
\begin{split}
&\Big|\int_{|\xi|\leq1}e^{i((x-y)\cdot\xi+(\phi(t)-\phi(s))|\xi|)}\phi(t)^{-\frac{m}{2(m+2)}}\frac{\beta(\delta\xi)}{|\xi|}\md\xi\Big| \\
&\le C\big(1+\big|\phi(t)-\phi(s)\big|\big)^{-\frac{n-1}{2}}\phi(t)^{-\frac{m}{2(m+2)}} \\
&\le C\big(1+|x-y|\big)^{-\frac{n-1}{2}}\phi(t)^{-\frac{m}{2(m+2)}}.
\end{split}
\end{equation}
In the last step of \eqref{equ:5.15} we have used the fact $\phi(t)-\phi(s)\geq|x-y|$ for any $(s,y)\in \operatorname{supp}F$ and $(t,x)\in \operatorname{supp}w$.
In fact, it follows from
the formula in Theorem 2.4 of \cite{Yag3} that the solution of \eqref{equ:3.1} satisfies
\begin{equation}
\begin{split}
w(t,x)=C_m\int_0^t\int_0^{\phi(t)-\phi(s)}&\partial_{r_1}\psi_{(F)_s}(r_1,x)\big(\phi(t)+\phi(s)+r_1\big)^{-\gamma}
\big(\phi(t)+\phi(s)-r_1\big)^{-\gamma}\\
&\times H\big(\gamma,\gamma,1,z\big)dr_1ds,
\end{split}
\label{equ:7.1}
\end{equation}
where $z=\frac{(-r_1+\phi(t)-\phi(s))(-r_1-\phi(t)+\phi(s))}
{(-r_1+\phi(t)+\phi(s))(-r_1-\phi(t)-\phi(s))}$, $H\big(\gamma,\gamma,1,z\big)$
is the hypergeometric function, and \\
$\psi_g(r_1,x)$ stands for the solution of the linear wave equation
\begin{equation*}
\begin{cases}
\partial_{r_1}^2 \psi-\Delta \psi=0, \\
\psi(0,x)=0, \quad \partial_{r_1} \psi(0,x)=g(x).\\
\end{cases}
\end{equation*}
Then by the finite propagation speed property for the linear wave operator and the expression \eqref{equ:7.1},
we  have that for any $(s,y)\in \operatorname{supp}F$ and $(t,x)\in \operatorname{supp}w$,
\[\phi(t)-\phi(s)\geq|x-y|.\]
Note that the corresponding inequality \eqref{equ:5.7} holds when we replace $v$ by
\[v_{01}=\int_0^t\int_{\mathbb{R}^n}\int_{|\xi|\leq1}e^{i((x-y)\cdot\xi+(\phi(t)-\phi(s))|\xi|)}\phi(t)^{-\frac{m}{2(m+2)}}
\frac{\beta(\delta\xi)}{|\xi|}G(s,y)\md y\md\xi \md s.\]
It follows from direct computation that
\begin{equation}\label{equ:5.16}
\begin{split}
\parallel &v_{01}\parallel_{L^2(\{(t,x): \frac{T}{2}\leq t\leq T, \delta\leq\phi(t)-|x|\leq2\delta\})} \\
&\leq\Big\|\iiint_{|\xi|\leq1}e^{i((x-y)\cdot\xi+(\phi(t)-\phi(s))|\xi|)}\phi(t)^{-\frac{m}{2(m+2)}}\frac{\beta(\delta\xi)}{|\xi|}\md\xi
G(s,y)\md y\md s\Big\|_{L^2_{t,x}} \\
&\leq C\Big\|\iint\big(1+\big|x-y\big|\big)^{-\frac{n-1}{2}}\phi(t)^{-\frac{m}{2(m+2)}}G(s,y)\md y\md s\Big\|_{L^2_{t,x}} \\
&\leq C\phi(T)^{-\frac{m}{2(m+2)}}\Big\| \Big(\parallel\big(1+\big|x-y\big|\big)^{-\frac{n-1}{2}}\parallel_{L^2_{s,y}}\parallel G\parallel_{L^2}\Big)\Big\|_{L^2_{t,x}} \\
&\leq C\phi(T)^{-\frac{m}{2(m+2)}}\parallel G\parallel_{L^2}\Big\|\Big(\int_{\frac{T}{2}}^T\int_{\delta\leq\phi(t)-|x|\leq2\delta}
\big(1+|x-y|\big)^{-(n-1)}\md x\md t\Big)^{\frac{1}{2}}\Big\|_{L^2_{s,y}}.
\end{split}
\end{equation}
By $|y|\leq\phi(2)$, then $\frac{1}{2}|x|\leq|x-y|\leq2|x|$ holds if $|x|\geq2\phi(2)$. On the other hand, if $|x|<2\phi(2)$,
then the integral with respect to the variable $x$ in last line of  in \eqref{equ:5.16}
must be finite and can be controlled by $\delta$. This yields
\begin{equation*}
\begin{split}
\Big\|&\Big(\int_{\frac{T}{2}}^T\int_{\delta\leq\phi(t)-|x|\leq2\delta}
(1+\big|x-y\big|)^{-(n-1)}\md x\md t\Big)^{\frac{1}{2}}\Big\|_{L^2_{s,y}} \\
&\leq C\Big\|\Big(\int_{\frac{T}{2}}^T\int_{\phi(t)-2\delta}^{\phi(t)-\delta}\md r\md t\Big)^{\frac{1}{2}}\Big\|_{L^2_{s,y}} \\
&\leq C(\delta_0\delta T)^{\frac{1}{2}},
\end{split}
\end{equation*}
which implies that the left side of \eqref{equ:5.7} can be controlled by
\begin{equation}
\begin{split}
\phi(T)^{\frac{m-2}{2(m+2)}-\frac{\nu}{2}}\delta^{\frac{m-2}{2(m+2)}+\frac{\nu}{2}}
\phi(T)^{-\frac{m}{2(m+2)}}(\delta_0\delta T)^{\frac{1}{2}}\parallel G\parallel_{L^2}\leq C\delta_0^{\frac{1}{2}}\parallel G\parallel_{L^2(\mathbb{R}^{n+1}_+)}.
\end{split}
\label{equ:A}
\end{equation}
Consequently, the proof on \eqref{equ:5.7}  will be completed if we could show that
\begin{equation}
\begin{split}
&\phi(T)^{\frac{m-2}{2(m+2)}-\frac{\nu}{2}}\delta^{\frac{m-2}{2(m+2)}+\frac{\nu}{2}}\parallel v_{02}\parallel_{L^2(\{(t,x):\frac{T}{2}\leq t\leq T, \delta\leq\phi(t)-|x|\leq2\delta\})} \\
&\leq C\delta_0^{\frac{1}{2}}\parallel G\parallel_{L^2(\mathbb{R}^{n+1}_+)},
\end{split}
\label{equ:5.10}
\end{equation}
where
\[v_{02}=\int_0^t\int_{\mathbb{R}^n}\int_{|\xi|\geq1}e^{i((x-y)\cdot\xi+(\phi(t)-\phi(s))|\xi|)}\phi(t)^{-\frac{m}{2(m+2)}}
\frac{\beta(\delta\xi)}{|\xi|}G(s,y)\md y\md\xi \md s.\]
The first step in proving \eqref{equ:5.10} is to notice that
\begin{equation*}
\begin{split}
&\parallel v_{02}\parallel_{L^2(\{(t,x): \frac{T}{2}\leq t\leq T, \delta\leq\phi(t)-|x|\leq2\delta\})} \\
&\leq\Big\| \int\parallel\check{T}G\parallel_{L^2(\{x:\delta\leq\phi(t)-|x|\leq2\delta\})}ds
\Big\|_{L^2(\{t:\frac{T}{2}\leq t\leq T\})},
\end{split}
\end{equation*}
where
\[\check{T}G=\int\int_{|\xi|\geq1}e^{i((x-y)\cdot\xi+(\phi(t)-\phi(s))|\xi|)}\phi(t)^{-\frac{m}{2(m+2)}}
\frac{\beta(\delta\xi)}{|\xi|}G(s,y)\md y\md\xi.\]
To estimate $\parallel\check{T}G\parallel_{L^2(\{x:\delta\leq\phi(t)-|x|\leq2\delta\})}$,
it follows from Lemma \ref{lem:a1} and direct computation that
\begin{equation}
\begin{split}
&\phi(T)^{\frac{m-2}{2(m+2)}-\frac{\nu}{2}}\delta^{\frac{m-2}{2(m+2)}+\frac{\nu}{2}}\parallel v_{02}\parallel_{L^2(\{(t,x):\frac{T}{2}\leq t\leq T, \delta\leq\phi(t)-|x|\leq2\delta\})} \\
&\leq C\phi(T)^{\frac{m-2}{2(m+2)}-\frac{\nu}{2}}\delta^{\frac{m-2}{2(m+2)}+\frac{\nu}{2}}\delta^{\frac{1}{2}}T^{\frac{1}{2}}
\phi(T)^{-\frac{m}{2(m+2)}} \\
&\qquad \times \Big(\sum_{j=0}^{\infty}2^{\frac{j}{2}}
\Big\|\iint e^{i(-y\cdot\xi-\phi(s))|\xi|)}\frac{\beta(\delta\xi)}{|\xi|}G(s,y)\md y\md s\Big\|_{L^2(2^j\leq|\xi|\leq2^{j+1})}\Big) \\
&\leq C\delta^{\frac{m}{m+2}}\Big(\sum_{j=0}^{\infty}\Big\|\iiint_{2^j\leq|\xi|\leq2^{j+1}} e^{i((x-y)\cdot\xi-\phi(s))|\xi|)}\frac{\beta(\delta\xi)}{|\xi|^{\frac{1}{2}}}G(s,y)\md y\md\xi \md s\Big\|_{L^2_x}\Big).
\end{split}
\label{equ:5.11}
\end{equation}
Then applying H\"{o}lder's inequality as in \eqref{equ:4.55} yields
\begin{equation}
\begin{split}
&\Big\|\iiint_{2^j\leq|\xi|\leq2^{j+1}} e^{i((x-y)\cdot\xi-\phi(s))|\xi|)}\frac{\beta(\delta\xi)}{|\xi|^{\frac{1}{2}}}G(s,y)\md y \md\xi \md s\Big\|_{L^2_x} \\
&\leq C\delta_0^{\frac{1}{2}}\Big(\iint\Big|\iint_{2^j\leq|\xi|\leq2^{j+1}}e^{i((x-y)\cdot\xi-(|y|+\tau)|\xi|)}
\frac{\beta(\delta\xi)}{|\xi|^{\frac{1}{2}}}
G(\phi^{-1}(|y|+\tau),y)\md y\md\xi\Big|^2\md x\md\tau\Big)^{\frac{1}{2}}.
\end{split}
\label{equ:5.12}
\end{equation}
On the other hand, by applying Lemma 3.2 in \cite{Gls}, we obtain that
for each fixed $j\ge 0$,
\begin{equation}
\begin{split}
&\Big\|\iint_{2^j\leq|\xi|\leq2^{j+1}}e^{i((x-y)\cdot\xi-(|y|+\tau)|\xi|)}\frac{\beta(\delta\xi)}{|\xi|^{\frac{1}{2}}}
G(\phi^{-1}(|y|+\tau),y)\md y\md\xi\Big\|_{L^2_{\tau,x}} \\
&\leq C\Big(\iint_{\phi(1)\leq|y|+\tau\leq \phi(2)}\left|G\big(\phi^{-1}(|y|+\tau),y\big)\right|^2\md y\md\tau\Big)^{\frac{1}{2}}\\
&=C\Big(\iint_{\phi(1)\leq|y|+\tau\leq \phi(2)}\left|G\big(\phi^{-1}(|y|+\tau),y\big)\right|^2\md\tau \md y\Big)^{\frac{1}{2}} \\
&\leq C\Big(\int_{\mathbb{R}^n}\int_1^2|G(s,y)|^2s^{\frac{m}{2}}\md s\md y\Big)^{\frac{1}{2}} \\
&\leq C\Big(\int_{\mathbb{R}^n}\int_1^2|G(s,y)|^2\md s\md y\Big)^{\frac{1}{2}} \\
&\leq C\parallel G\parallel_{L^2{(\mathbb{R}^{n+1}_+)}}. \\
\end{split}
\label{equ:5.13}
\end{equation}
On the other hand, in the support of $\beta(\delta\xi)$, one has $2^j\delta\leq|\xi|\delta\leq C$, which derives
\begin{equation}
\begin{split}
&j\leq C(1+|\ln{\delta}|).
\end{split}
\label{equ:5.14}
\end{equation}
Substituting \eqref{equ:5.13} and \eqref{equ:5.14} into \eqref{equ:5.12} and further \eqref{equ:5.11}, we arrive at
%\newpage

\begin{align*}
&\phi(T)^{\frac{m-2}{2(m+2)}-\frac{\nu}{2}}\delta^{\frac{m-2}{2(m+2)}+\frac{\nu}{2}}\parallel v_{02}\parallel_{L^2(\{(t,x):\frac{T}{2}\leq t\leq T, \delta\leq\phi(t)-|x|\leq2\delta\})} \\
&\leq C\delta^{\frac{m}{m+2}}(1+|\ln{\delta}|)\delta_0^{\frac{1}{2}}\parallel G\parallel_{L^2{(\mathbb{R}^{n+1}_+)}} \\
&\leq C\delta_0^{\frac{1}{2}}\parallel G\parallel_{L^2{(\mathbb{R}^{n+1}_+)}}.
\end{align*}
%Hence we eventually complete the proof of \eqref{equ:3.4}.
%\qquad \qquad \qquad \qquad \qquad \qquad \qquad \qquad \qquad \qquad \qquad $\square$

\subsection{Estimate of \(w^1\)}
We only need to prove \eqref{equ:5.7} for \(\phi(T)\leq10\phi(2)\). In this case
we must have \(\delta_0\leq\delta\leq10\phi(2)\). Our task is reduced to prove
\begin{equation}
\delta^{\frac{m-2}{2(m+2)}+\frac{\nu}{2}}\parallel v\parallel_{L^2(\{(t,x):\frac{T}{2}\leq t\leq T, \delta\leq\phi(t)-|x|\leq2\delta\})}\leq C\delta_0^{\frac{1}{2}}\parallel G\parallel_{L^2(\mathbb{R}^{n+1}_+)}.
\label{equ:5.17}
\end{equation}
As in Subsection \ref{sec4:w0:small}, we can write
\[v=\int_0^t\int_{\mathbb{R}^n}\int_{\mathbb{R}^n}e^{i((x-y)\cdot\xi+(\phi(t)-\phi(s))|\xi|)}
|\xi|^{-1}G(s,y)\md y\md\xi \md s.\]
Note in this case the time variable \(t\) is bounded from both below and above, thus we do not require
to focus on the term \(\phi(t)^{-\frac{m}{2(m+2)}}\)in Subsection 5.1. As in Subsection \ref{sec4:w0:small},
we shall split $v$ into a low frequency  part and a high frequency part respectively.
To this end, we choose a function $\beta\in C_0^\infty(\mathbb{R}^n)$ satisfying $\beta=1$ near the origin
such that $v=v_0+v_1$, where
\begin{equation*}
\begin{split}
v_1&=\int_0^t\int_{\mathbb{R}^n}\int_{\mathbb{R}^n}e^{i((x-y)\cdot\xi+(\phi(t)-\phi(s))|\xi|)}
\frac{1-\beta(\delta\xi)}{|\xi|}G(s,y)\md y\md\xi \md s. \\
%&=\int_0^tT_1(t,s,x)\md s.
\end{split}
\end{equation*}
If we set \(\phi(s)=|y|+\tau\) and use H\"{o}lder's inequality as in \eqref{equ:4.55}, then
\begin{equation*}
\begin{split}
|v_1|&\leq C\delta_0^{\f{1}{2}}\bigg(\int_{\delta_0}^{2\delta_0}\bigg|\int_{\mathbb{R}^n}\int_{\mathbb{R}^n}e^{i((x-y)\cdot\xi+(\phi(t)-\tau-|y|)|\xi|)}
\frac{1-\beta(\delta\xi)}{|\xi|}G(s,y)\md y\md\xi\bigg|^2
\md \tau\bigg)^{\f{1}{2}}\\
&=:C\delta_0^{\f{1}{2}}\bigg(\int_{\delta_0}^{2\delta_0}|\bar{T_1}(t,\tau,\cdot)|^2\md \tau\bigg)^{\f{1}{2}}.
\end{split}
\end{equation*}
Note $\frac{1-\beta(\delta\xi)}{|\xi|}=O(\delta)$. Then the expression of $v_1$ is similar to \eqref{equ:4.58} with $Rez=0$
. Consequently we can apply the method of \eqref{equ:4.60} to get
\[\parallel \bar{T_1}(t,\tau,\cdot)\parallel_{L^2(\mathbb{R}^n)}\leq C\big(\phi(t)-\tau\big)^{\frac{\nu}{2}}
\delta^{-\frac{\nu+1}{2}+1}\parallel G(s,\cdot)\parallel_{L^2(\mathbb{R}^n)},\]
which derives
\[\parallel v_1\parallel_{L^2}\leq C\delta_0^{\frac{1}{2}}\delta^{-\frac{\nu+1}{2}+1}\phi(T)^{\frac{\nu}{2}}\parallel G\parallel_{L^2(\mathbb{R}^{n+1}_+)}.\]
Due to $\delta\leq\phi(T)\leq10\phi(2)$ and \(\phi(T)\geq\phi(1)\), the estimate \eqref{equ:5.17} for $v_1$ follows immediately.

We now estimate $v_0$. At first, one notes that similarly to \eqref{equ:5.15}, we have
\begin{equation*}
\begin{split}
\bigg|\int_{|\xi|\leq1}e^{i((x-y)\cdot\xi+(\phi(t)-\phi(s))|\xi|)}\frac{\beta(\delta\xi)}{|\xi|}\md\xi\bigg|
\le C\big(1+\big|x-y\big|\big)^{-\frac{n-1}{2}}.
\end{split}
\end{equation*}
Thus the corresponding inequality \eqref{equ:5.7} holds as long as we replace $v$ by
\[v_{01}=\int_0^t\int_{\mathbb{R}^n}\int_{|\xi|\leq1}e^{i((x-y)\cdot\xi+(\phi(t)-\phi(s))|\xi|)}
\frac{\beta(\delta\xi)}{|\xi|}G(s,y)\md y\md\xi \md s.\]
As in \eqref{equ:5.16} one has
\begin{equation*}
\begin{split}
\parallel &v_{01}\parallel_{L^2(\{(t,x): \frac{T}{2}\leq t\leq T, \delta\leq\phi(t)-|x|\leq2\delta\})} \\
&\leq\Big\|\iiint_{|\xi|\leq1}e^{i((x-y)\cdot\xi+(\phi(t)-\phi(s))|\xi|)}\frac{\beta(\delta\xi)}{|\xi|}\md\xi
G(s,y)\md y\md s\Big\|_{L^2_{t,x}} \\
&\leq C\parallel G\parallel_{L^2}\Big\|\Big(\int_{\frac{T}{2}}^T\int_{\delta\leq\phi(t)-|x|\leq2\delta}
(1+\big|x-y\big|)^{-(n-1)}\md x\md t\Big)^{\frac{1}{2}}\Big\|_{L^2_{s,y}}.
\end{split}
\end{equation*}
Note that in this case of \(|x|\leq\phi(t)\leq10\phi(2)\), a direct computation yields
\begin{equation*}
\begin{split}
\Big\|&\Big(\int_{\frac{T}{2}}^T\int_{\delta\leq\phi(t)-|x|\leq2\delta}
\big(1+\big|x-y\big|\big)^{-(n-1)}\md x\md t\Big)^{\frac{1}{2}}\Big\|_{L^2_{s,y}} \\
&\leq C\Big\|\Big(\int_{\frac{T}{2}}^T\int_{\phi(t)-2\delta}^{\phi(t)-\delta}\md r\md t\Big)^{\frac{1}{2}}\Big\|_{L^2_{s,y}} \\
&\leq C(\delta_0\delta T)^{\frac{1}{2}},
\end{split}
\end{equation*}
which implies that the left side of \eqref{equ:5.17} can be controlled by
\begin{equation}
\begin{split}
\delta^{\frac{m-2}{2(m+2)}+\frac{\nu}{2}}
(\delta_0\delta T)^{\frac{1}{2}}\parallel G\parallel_{L^2}\leq C\delta_0^{\frac{1}{2}}\parallel G\parallel_{L^2(\mathbb{R}^{n+1}_+)}.
\end{split}
\label{equ:5.18}
\end{equation}
Consequently, our proof will be completed if we could show that
\begin{equation*}
\begin{split}
&\delta^{\frac{m-2}{2(m+2)}+\frac{\nu}{2}}\parallel v_{02}\parallel_{L^2(\{(t,x):\frac{T}{2}\leq t\leq T, \delta\leq\phi(t)-|x|\leq2\delta\})} \\
&\leq C\delta_0^{\frac{1}{2}}\parallel G\parallel_{L^2(\mathbb{R}^{n+1}_+)},
\end{split}
%\label{equ:5.10}
\end{equation*}
where
\[v_{02}=\int_0^t\int_{\mathbb{R}^n}\int_{|\xi|\geq1}e^{i((x-y)\cdot\xi+(\phi(t)-\phi(s))|\xi|)}
\frac{\beta(\delta\xi)}{|\xi|}G(s,y)\md y\md\xi \md s.\]
But this just only follows from the estimate of \(v_{02}\) in Subsection \ref{sec4:w0:small} if one notes
the fact of \(\delta\leq\phi(T)\leq10\phi(2)\).

Collecting the estimates on $w^0$ and $w^1$ in Subsection 5.1 and Subsection 5.2 respectively,
the proof of (3.4) is completed.

\section{Proof of Theorem 1.1.}\label{sec:6}
Based on the weighted inequalities $\S 3$-$\S5$, we begin the proof of Theorem 1.1.

\begin{proof}
By the local existence and regularity of solution $u$ to (1.2)
(for examples, one can see \cite{Rua1} or references therein), one knows that $u\in C^{\infty}([0, T_0]\times\Bbb R^n)$
exists for any fixed constant $T_0<1$ and $u$ has a compact support on the variable $x$. Moreover, for any $N\in\Bbb N$,
\begin{equation}
\begin{split}
&\Big\|{u}\Big(\f{T_0}{2}, \cdot\Big)\Big\|_{C^N}+\Big\| \partial_t{u}\Big(\f{T_0}{2}, \cdot\Big)\Big\|_{C^N}\le C_{N}\ve.\\
\end{split}
\label{equ:6.1}
\end{equation}
Then we can take $({u}(\f{T_0}{2}, x), \partial_t{u}(\f{T_0}{2}, x))$ as the new initial data to
solve (1.2) from $t=\f{T_0}{2}$.

Now we use the standard Picard iteration to prove Theorem 1.1. Let $u_{-1}\equiv0$, and for $k=0,1,2,3,\ldots$, let $u_k$ be the solution of
the following equation
\begin{equation*}%\label{equ:6.2}
\begin{cases}
&\partial_t^2 u_k-t^m\triangle u_k=F_p(t,u_{k-1}), \quad (t,x)\in\big(\f{T_0}{2}, \infty\big)\times\mathbb{R}^n,\\
&u_k\big(\f{T_0}{2},x\big)={u}\big(\f{T_0}{2},x\big)\quad \partial_tu_k\big(\f{T_0}{2},x\big)=\partial_t{u}\big(\f{T_0}{2}, x\big).
\end{cases}
\end{equation*}
For  $p\in (p_{crit}(m,n), p_{conf}(m,n))$, we can fix a number $\gamma>0$ satisfying
\[\frac{1}{p(p+1)}<\gamma<\frac{(m+2)n-2}{2(m+2)}-\frac{(m+2)n-m}{(m+2)(p+1)}.\]
Set
\begin{align*}
M_k=&\Big\|\Big(\big(\phi(t)+M\big)^2-|x|^2\Big)^\gamma u_k\Big\|_{L^q([\f{T_0}{2},\infty)\times\mathbb{R}^n)}, \\
N_k=&\Big\|\Big(\big(\phi(t)+M\big)^2-|x|^2\Big)^\gamma (u_k-u_{k-1})\Big\|_{L^q([\f{T_0}{2},\infty)\times\mathbb{R}^n)},
\end{align*}
where $q=p+1$. By \eqref{equ:6.1} and Theorem 2.1 we know that there exists a constant $C_0>0$ such that
\[M_0\leq C_0\ve.\]
Notice that for $j$, $k\geq0$,
\begin{equation*}
\begin{cases}
&\partial_t^2 (u_{k+1}-u_{j+1})-t^m \Delta (u_{k+1}-u_{j+1}) =V(u_k,u_j)(u_k-u_j),   \\
&(u_{k+1}-u_{j+1})\big(\f{T_0}{2}, x\big)=0, \quad \partial_t(u_{k+1}-u_{j+1})\big(\f{T_0}{2}, x\big)=0,
\end{cases}
\end{equation*}
where
\begin{equation*}%\label{equ:6.3}
\big|V(u_k,u_j)\big|\le
\begin{cases}
&C(|u_k|+|u_j|)^{p-1} \quad \text{if} \quad t\geq T_0, \\
&C(1+|u_k|+|u_j|)^{p-1} \quad \text{if} \quad \f{T_0}{2}\leq t\leq T_0.
\end{cases}
\end{equation*}
By our assumptions
\[\gamma<\frac{(m+2)n-2}{2(m+2)}-\frac{(m+2)n-m}{(m+2)q} \quad \text{and} \quad p\gamma>\frac{1}{q}, \quad q=p+1,\]
then applying Theorem 3.2 and H\"{o}lder's inequality yields
\begin{equation}\label{equ:6.4}
\begin{split}
&\Big\|\Big(\big(\phi(t)+M\big)^2-|x|^2\Big)^\gamma (u_{k+1}-u_{j+1})\Big\|_{L^q([\f{T_0}{2},\infty)\times\mathbb{R}^n)} \\
&\leq C\Big\|\Big(\big(\phi(t)+M\big)^2-|x|^2\Big)^{p\gamma}V(u_k,u_j)(u_k-u_j)\Big\|_{L^{\frac{q}{q-1}}([\f{T_0}{2},\infty)\times\mathbb{R}^n)} \\
&\leq C\Big\{\Big\|\Big(\big(\phi(t)+M\big)^2-|x|^2\Big)^\gamma (1+|u_k|+|u_j|)\Big\|_{L^q([\f{T_0}{2}, T_0]\times\mathbb{R}^n)} \\
&\qquad +\Big\|\Big(\big(\phi(t)+M\big)^2-|x|^2\Big)^\gamma (|u_k|+|u_j|)\Big\|_{L^q([T_0,\infty]\times\mathbb{R}^n)}\Big\}^{p-1} \\
&\quad \times \Big\|\Big(\big(\phi(t)+M\big)^2-|x|^2\Big)^\gamma (u_k-u_j)\Big\|_{L^q([\f{T_0}{2},\infty)\times\mathbb{R}^n)} \\
&\leq C\big(C_1T_0^{\frac{1}{q}}+M_k
+M_j\big)^{p-1}\Big\|\Big(\big(\phi(t)+M\big)^2-|x|^2\Big)^\gamma (u_k-u_j)\Big\|_{L^q([\f{T_0}{2},\infty)\times\mathbb{R}^n)}.
\end{split}
\end{equation}
If $j=-1$, then $M_j=0$, and we conclude that from \eqref{equ:6.4}
\[M_{k+1}\leq M_0+\frac{M_k}{2}\quad \text{for} \quad C\big(C_1T_0^{\frac{1}{q}}+M_k\big)^{p-1}\leq\frac{1}{2}.\]
This yields that
\[M_k\leq 2M_0 \quad \text{if} \quad C\big(C_1T_0^{\frac{1}{q}}+C_0\ve\big)^{p-1}\leq\frac{1}{2}.\]
Thus we get the boundedness of $\{u_k\}$ in the space $L^q(\Bbb R_+^{n+1})$ when the fixed constant $T_0$
and $\ve>0$ are sufficiently small. Similarly, we have
\[N_{k+1}\leq\frac{1}{2}N_k,\]
which derives that there exists a function $u\in L^q\big(\big[\f{T_0}{2}, \infty\big)\times\mathbb{R}^n\big)$ such that  $u_k\rightarrow u\in L^q\big(\big[\f{T_0}{2}, \infty\big)\times\mathbb{R}^n\big)$.
In addition, by the uniform boundedness of $M_k$ and the computations above, one easily obtains
\begin{align*}
&\parallel F_p(t,u_{k+1})-F_p(t,u_k)\parallel_{L^{\frac{q}{q-1}}([\f{T_0}{2}, \infty)\times\mathbb{R}^n)} \\
&\leq C\parallel u_{k+1}-u_k\parallel_{L^q([\f{T_0}{2}, \infty)\times\mathbb{R}^n)} \\
&\leq C\phi\Big(\frac{T_0}{4}\Big)^{-\gamma}N_k \\
&\le C2^{-k}.
\end{align*}
Therefore $F_p(t,u_k)\rightarrow F_p(t,u)$ in $L^{\frac{q}{q-1}}\big(\big[\f{T_0}{2}, \infty\big)\times\mathbb{R}^n\big)$ and hence $u$ is a weak solution of (1.2)
in the sense of distributions. Then we complete the proof of Theorem 1.1.
\end{proof}

\appendix
\section{Appendix.}\label{app}
The first lemma is from Lemma 3.8 of \cite{Gl2}.
\begin{lemma}\label{lem:a5}
Assume that  $\beta(\tau)\in C_0^\infty(\frac{1}{2},2)$ and
$\sum\limits_{j=-\infty}^\infty\beta(\frac{\tau}{2^j}) \equiv1$ for $\tau>0$.
Define the Littlewood-Paley operators of function $G$ as follows
\[G_j(t,x)=(2\pi)^{-n}\int_{\mathbb{R}^n}e^{ix\cdot\xi}\beta(\frac{|\xi|}{2^j})\hat{G}(t,\xi)d\xi.\]
Then one has that
\begin{align*}
\parallel G\parallel_{L^s_tL^q_x}\leq C(\sum\limits_{j=-\infty}^{\infty}\parallel G_j\parallel^2_{L^s_tL^q_x})^{\frac{1}{2}}
\qquad \text{for\quad $2\leq q<\infty$ and $2\leq s \leq \infty$}
\end{align*}
and
\begin{align*}
(\sum\limits_{j=-\infty}^{\infty}\parallel G_j\parallel^2_{L^r_tL^p_x})^{\frac{1}{2}}\leq C\parallel G\parallel_{L^r_tL^p_x}
\qquad \text{for\quad $1<p\leq2$\quad and \quad $1\leq r \leq 2$}.
\end{align*}
\end{lemma}

The following lemma is a variant of Lemma 3.2 in \cite{Gls}.

\begin{lemma}\label{lem:a2}
If \eqref{equ:4.57} holds, then \eqref{equ:4.39} holds.
\end{lemma}

\begin{proof}
If \eqref{equ:4.57} holds, then we arrive at
\begin{equation*}
\begin{split}
&\parallel v\parallel_{L^{q_0}(\{(t,x): \delta\leq\phi(t)-|x|\leq2\delta\})} \\
&\leq C\delta_0^{\frac{1}{q_0}}\Big\|\Big(\int_{\delta_0}^{2\delta_0}
\parallel (TG)\big(\phi(t)-\tau,\cdot\big)\parallel_{L^{q_0}(\{x:\delta\leq\phi(t)-|x|\leq2\delta\})}^{\frac{q_0}{q_0-1}}\md\tau\Big)^{\frac{q_0-1}{q_0}}
\Big\|_{L^{q_0}_t} \\
&\leq C\delta_0^{\frac{1}{q_0}}\bigg\|\bigg(\int_{\delta_0}^{2\delta_0}\Big(\big(\phi(t)-\tau\big)^{\nu-\frac{m+4}{q_0(m+2)}}\delta^{-\nu-\frac{1}{q_0}}
\parallel G(\tau,\cdot)\parallel_{L^{\frac{q_0}{q_0-1}}}\Big)^{\frac{q_0}{q_0-1}}\md\tau\bigg)^{\frac{q_0-1}{q_0}}\bigg\|_{L^{q_0}_t} \\
&\leq C\delta_0^{\frac{1}{q_0}}\Big\|\phi(t)^{\nu-\frac{m+4}{q_0(m+2)}}\delta^{-\nu-\frac{1}{q_0}}\Big(\int_{\delta_0}^{2\delta_0}
\parallel G(\tau,\cdot)\parallel_{L^{\frac{q_0}{q_0-1}}}^{\frac{q_0}{q_0-1}}\md\tau\Big)^{\frac{q_0-1}{q_0}}\Big\|_{L^{q_0}_t} \\
&\leq C\delta_0^{\frac{1}{q_0}}\delta^{-\nu-\frac{1}{q_0}}\Big(\int_{\frac{T}{2}}^T\phi(t)^{q_0\nu-\frac{m+4}{m+2}}\md t\Big)^{\frac{1}{q_0}}
\parallel G\parallel_{L^{\frac{q_0}{q_0-1}}} \\
&\leq C\delta_0^{\frac{1}{q_0}}\delta^{-\nu-\frac{1}{q_0}}\phi(T)^{\nu-\frac{1}{q_0}}\parallel G\parallel_{L^{\frac{q_0}{q_0-1}}},
\end{split}
\end{equation*}
which derives  \eqref{equ:4.39}.
\end{proof}

\begin{lemma}\label{lem:a3}
\eqref{equ:4.66} holds.
\end{lemma}

\begin{proof}
We will apply Lemma 3.2 in \cite{Gls} and the dual argument to derive \eqref{equ:4.66}.
For $h\in L^2(\mathbb{R}^n)$, due to $(R_z^*h,\bar{g})=(h,\bar{R_zg})$, then
\[\parallel R_zg\parallel_{L^2}=\sup_{h\in L^2}\frac{\big|(h,\bar{R_zg})\big|}{\parallel h\parallel_{L^2}}\leq\sup_{h\in L^2}\frac{\parallel R_z^*h\parallel_{L^2}}{\parallel h\parallel_{L^2}}\parallel g\parallel_{L^2}.\]
Since
\begin{equation*}
\begin{split}
(R_zg)(t,x)=&\Big(z-\frac{(m+2)n+2}{2(m+2)}\Big)e^{z^2} \int_{\mathbb{R}^n}\int_{\frac{1}{2}\phi(1)\leq|y|\leq\phi(2)}e^{i((x-y)\cdot\xi-(\phi(t)-|y|)|\xi|)} \\
&\quad \times \Big(1-\rho\big(\phi(t)^{1-\alpha}\delta^\alpha\xi\big)\Big)\big(1+\phi(t)|\xi|\big)^{-\frac{m}{2(m+2)}}g(y)\frac{\md\xi}{|\xi|^z}\md y,
\end{split}
\end{equation*}
the dual operator of $R_zg$ is
\begin{equation*}
\begin{split}
(R_z^*h)&(y)=\Big(\bar{z}-\frac{(m+2)n+2}{2(m+2)}\Big)e^{\bar{z}^2}\int\int e^{i((y-x)\cdot\xi+(\phi(t)-|y|)|\xi|)} \\
&\qquad \times\Big(1-\rho\big(\phi(t)^{1-\alpha}\delta^\alpha\xi\big)\Big)\big(1+\phi(t)|\xi|\big)^{-\frac{m}{2(m+2)}}h(x)\frac{\md\xi}{|\xi|^{\bar{z}}}\md x\md\xi \\
&=(\bar{z}-\frac{(m+2)n+2}{2(m+2)})e^{\bar{z}^2}\int e^{i(y\cdot\xi-|y||\xi|)}e^{i\phi(t)|\xi|} \\
&\qquad \times\Big(1-\rho\big(\phi(t)^{1-\alpha}\delta^\alpha\xi\big)\Big)\big(1+\phi(t)|\xi|\big)^{-\frac{m}{2(m+2)}}
|\xi|^{-\bar{z}}\hat{h}(\xi)\md\xi.
\end{split}
\end{equation*}
Denote
\[\hat{H}(\xi)=e^{i\phi(t)|\xi|}\Big(1-\rho\big(\phi(t)^{1-\alpha}\delta^\alpha\xi\big)\Big)\big(1+\phi(t)|\xi|\big)^{-\frac{m}{2(m+2)}}|\xi|^{-\bar{z}}\hat{h}(\xi).\]
Then it follows from direct computation that
\begin{equation}\label{equ:a.1}
\begin{split}
&\parallel R_z^*g\parallel_{L^2(\frac{1}{2}\phi(1)\leq|y|\leq\phi(2))} \\
&\leq C\Big(\parallel\hat{H}\parallel_{L^2(|\xi|\leq1)}
+\sum_{k=0}^{\infty}2^{\frac{k}{2}}\parallel\hat{H}\parallel_{L^2(2^k\leq|\xi|\leq2^{k+1})}\Big)\qquad
(\text{by Lemma 3.2 of \cite{Gls}}) \\
&\leq C\Big(\parallel\hat{H}\parallel_{L^2(|\xi|\leq1)}
+\sum_{2^{k+1}\leq\phi(t)^{\alpha-1}\delta^{-\alpha}}^{\infty}2^{\frac{k}{2}}\parallel\hat{H}\parallel_{L^2(2^k\leq|\xi|\leq2^{k+1})}\Big).
%&\leq C\delta^{-\frac{\alpha}{2}}\phi(t)^{\frac{\alpha-1}{2}}\parallel\hat{H}\parallel_{L^2} \\
%&\leq C\delta^{-\frac{\alpha}{2}}\phi(t)^{\frac{\alpha-1}{2}}\parallel\hat{h}\parallel_{L^2}.
\end{split}
\end{equation}
If $k\geq0$, then
\begin{equation}\label{equ:7.11}
\parallel\hat{H}\parallel_{L^2(2^k\leq|\xi|\leq2^{k+1})}\leq C2^{-k\frac{m}{2(m+2)}}\phi(t)^{-\frac{m}{2(m+2)}}\parallel\hat{h}\parallel_{L^2(2^k\leq|\xi|\leq2^{k+1})}.
\end{equation}
Next we estimate the remaining part $\parallel\hat{H}\parallel_{L^2(|\xi|\leq1)}$ in \eqref{equ:a.1}.
We first consider the special case, that is, $\hat{h}(\xi)$ is a polynomial. More specifically,
without loss of generality, we assume that $\hat{h}(\xi)=\xi_1^l$, \(l\in\mathbb{N}\).
Note that the integral domain $\{\xi: |\xi|\leq1\}$ is contained in the area $D=\{\xi:
|\xi_j|\leq1, j=1,...,n\}$. Set $D'=\{(\xi_2, ..., \xi_n): |\xi_j|\leq1, j=2,...,n\}$. Then we have that for $\hat{h}(\xi)=\xi_1^l$, \[\parallel\hat{H}\parallel_{L^2(|\xi|\leq1)}\le I=:\int_{D}\frac{\xi_1^{2l}}{\big(1+\phi(t)|\xi|\big)^{\frac{m}{m+2}}}\md\xi.\]
Note that $\xi=(\xi_1, \xi')$. Then
\begin{equation*}
  \begin{split}
    I & =\int_{D'}\md\xi'\int_{-1}^{1}\frac{\xi_1^{2l}}{(1+\phi(t)|\xi|)^{\frac{m}{m+2}}}\md\xi_1 \\
      & =\frac{2}{2l+1}\int_{D'}\frac{1}{(1+\phi(t)|\xi|)^{\frac{m}{m+2}}}\md\xi'\Big|_{\xi_1=1} \\
       &\qquad +\frac{m}{(2l+1)(m+2)}\int_{D'}\md\xi'\int_{-1}^{1}\frac{\xi_1^{2l+1}\frac{\xi_1}{|\xi|}\phi(t)}{(1+\phi(t)|\xi|)^{\frac{m}{m+2}+1}}
       \md\xi_1 \\
       & \leq\frac{2}{2l+1}\int_{D'}\frac{1}{(1+\phi(t))^{\frac{m}{m+2}}}\md\xi'
       +\frac{m}{(2l+1)(m+2)}\int_{D'}d\xi'\int_{-1}^{1}\frac{\xi_1^{2l}|\xi|\phi(t)}{(1+\phi(t)|\xi|)^{\frac{m}{m+2}+1}}\md\xi_1 \\
       & \leq\frac{2C_n}{2l+1}(1+\phi(t))^{-\frac{m}{m+2}}+\frac{m}{m+2}I,
  \end{split}
\end{equation*}
where $C_n=\int_{D'}d\xi'$ only depends on $n$. Thus we have
\begin{align*}
   & (1-\frac{m}{m+2})I\leq\frac{2C_n}{2l+1}(1+\phi(t))^{-\frac{m}{m+2}}=2(1+\phi(t))^{-\frac{m}{m+2}}\int_{D}\xi_1^{2l}d\xi.
 \end{align*}
 This yields
   \begin{align}\label{equ:a.2}
   & (\int_{D}\frac{\xi_1^{2l}}{(1+\phi(t)|\xi|)^{\frac{m}{m+2}}}d\xi)^{\frac{1}{2}}\leq C_m(1+\phi(t))^{-\frac{m}{2(m+2)}}
   \parallel\xi_1^l\parallel_{L^2(D)}.
\end{align}
Analogously, we can prove the same estimates \eqref{equ:a.2} for all the polynomial $\hat h(\xi)$:
\begin{align}\label{equ:a.3}
   & (\int_{D}\frac{|\hat h(\xi)|^2}{(1+\phi(t)|\xi|)^{\frac{m}{m+2}}}d\xi)^{\frac{1}{2}}\leq C_m(1+\phi(t))^{-\frac{m}{2(m+2)}}
   \parallel \hat h(\xi)\parallel_{L^2(D)}.
\end{align}
 For each fixed $t>0$ and nonzero $\hat h(\xi)\in L^2$, we can find a continuous function $g_t(\xi)$ such that \[\parallel\hat{h}-g_t\parallel_{L^2(D)}\leq\frac{1}{3}(1+\phi(t))^{-\frac{m}{2(m+2)}}\parallel\hat{h}\parallel_{L^2(D)}.\]
Since $D$ is compact, by Stone-Weierstrass theorem, we can find a polynomial $p_t(\xi)$ such that
\[\parallel p_t-g_t\parallel_{L^2(D)}\leq\frac{1}{3}(1+\phi(t))^{-\frac{m}{2(m+2)}}\parallel\hat{h}\parallel_{L^2(D)}.\]
Then we obtain the estimate
\begin{equation}\label{equ:7.12}
\parallel\hat{H}\parallel_{L^2(D)}\leq C(1+\phi(t))^{-\frac{m}{2(m+2)}}\parallel\hat{h}\parallel_{L^2(D)}.
\end{equation}
Combining \eqref{equ:7.11} and \eqref{equ:7.12} yields \eqref{equ:4.66}.
\end{proof}

\begin{lemma}\label{lem:a4}
\eqref{equ:4.67} holds true.
\end{lemma}

\begin{proof}
Denote $K_z$ by the kernel of the operator $S_z$. Then
\begin{equation}
\begin{split}
K_z(t;x,y)=&(z-\frac{(m+2)n+2}{2(m+2)})e^{z^2}\int_{\mathbb{R}^n}e^{i[(x-y)\cdot\xi-(\phi(t)-|y|)|\xi|]} \\ &\times\rho(\phi(t)^{1-\alpha}\delta^\alpha\xi)(1+\phi(t)|\xi|)^{-\frac{m}{2(m+2)}}\frac{d\xi}{|\xi|^z} \quad
\text{with $Rez=0$}.
\end{split}\label{equ:a.5}
\end{equation}
Note that $|\xi|\geq\phi(t)^{\alpha-1}\delta^{-\alpha}$ holds in the integral of
\eqref{equ:a.5}. Therefore it follows from Lemma 3.3 in \cite{Gls} and the assumption
of $\delta\leq10\phi(2)$ that for any $N\in\Bbb R^+$,
\begin{align*}
|K_z|&\le C_N\Big(\frac{\delta}{\phi(t)}\Big)^N(\phi(t)^\alpha\delta^{-\alpha})^{-\frac{m}{2(m+2)}}\leq C_N\Big(\frac{1}{\phi(t)}\Big)^N(\phi(t)^\alpha\delta^{-\alpha})^{-\frac{m}{2(m+2)}} \\
&\text{if}\quad \Big||x-y|-\big|\phi(t)-|y|\big|\Big|\geq\frac{\delta}{2}.
\end{align*}
This yields \eqref{equ:4.67} when $\Big||x-y|-\big|\phi(t)-|y|\big|\Big|\geq\frac{\delta}{2}$.
For the case of $\Big||x-y|-\big|\phi(t)-|y|\big|\Big|<\frac{\delta}{2}$,
analogously treated as in Lemma 3.4-Lemma 3.5 and Proposition 3.6 of \cite{Gls},
\eqref{equ:4.67} can be also derived, here we omit the detail since the proof procedure
is completely similar to that in \cite{Gls}.
\end{proof}

\begin{lemma}\label{lem:a1}
One has that for $\delta>0$,
\begin{equation}
\begin{split}
&\Big\|\int_{\Bbb R^n} e^{i(x\cdot\xi+t|\xi|)}\hat{f}(\xi)\md\xi\Big\|_{L^2(\{x:\delta\leq t-|x|\leq2\delta\})} \\
&\leq C\delta^{\frac{1}{2}}\Big(\|\hat{f}\parallel_{L^2(|\xi|\leq1)}
+\sum_{k=0}^{\infty}2^{\frac{k}{2}}\|\hat{f}\parallel_{L^2(2^k\leq|\xi|\leq2^{k+1})}\Big).
\end{split}\label{equ:a.6}
\end{equation}
\end{lemma}

\begin{proof}
Note that
\begin{equation*}
\begin{split}
&\Big\|\int_{\Bbb R^n}  e^{i(x\cdot\xi+t|\xi|)}\hat{f}(\xi)\md\xi\Big\|_{L^2(\{x:\delta\leq t-|x|\leq2\delta\})} \\
&=(t-\delta)^{-\frac{n}{2}}\Big\| \int e^{i(x\cdot\xi+\frac{t}{t-\delta}|\xi|)}\hat{f}\Big(\frac{\xi}{t-\delta}\Big)\md\xi\Big\|_{L^2(\{x:\frac{t-2\delta}{t-\delta}\leq |x|\leq1\})}.\\
\end{split}
\end{equation*}
By applying the Sobolev trace theorem to the function $\int e^{i(x\cdot\xi+\frac{t}{t-\delta}|\xi|)}\hat{f}(\frac{\xi}{t-\delta})d\xi$
for the variable $x$, we find that
\begin{equation}
\begin{split}
&\int_{\theta\in S^{n-1}}\Big|e^{i(r\theta\cdot\xi+\frac{t}{t-\delta}|\xi|)}\hat{f}\Big(\frac{\xi}{t-\delta}\Big)\md\xi\Big|^2\md\theta \\
&\leq C\Big(\Big\|\hat{f}\Big(\frac{\xi}{t-\delta}\Big)\Big\|_{L^2(|\xi|\leq1)}
+\sum_{j=0}^{\infty}2^{\frac{j}{2}}\Big\|\hat{f}\Big(\frac{\xi}{t-\delta}\Big)\Big\|_{L^2(2^j\leq|\xi|\leq2^{j+1})}\Big)^2.
\end{split}
\label{equ:B}
\end{equation}
Integrating  \eqref{equ:B} with respect to $r$ yields
\begin{equation*}
\begin{split}
&\Big\|\int_{\Bbb R^n}  e^{i(x\cdot\xi+t|\xi|)}\hat{f}(\xi)\md\xi\Big\|_{L^2(\{x:\delta\leq t-|x|\leq2\delta\})} \\
&\le \Big(\frac{t-2\delta}{t-\delta}\Big)^{\frac{n-1}{2}}\Big(\frac{\delta}{t-\delta}\Big)^{\frac{1}{2}}(t-\delta)^{-\frac{n}{2}}
\Big(\Big\|\hat{f}\Big(\frac{\xi}{t-\delta}\Big)\Big\|_{L^2(|\xi|\leq1)} \\
&\quad +\sum_{j=0}^{\infty}2^{\frac{j}{2}}\Big\|\hat{f}\Big(\frac{\xi}{t-\delta}\Big)\Big\|_{L^2(2^j\leq|\xi|\leq2^{j+1})}\Big) \\
\end{split}
\end{equation*}

\begin{equation*}
\begin{split}
&\leq C\Big(\frac{\delta}{t-\delta}\Big)^{\frac{1}{2}}(t-\delta)^{-\frac{n}{2}}
\Big((t-\delta)^{\frac{1}{2}}\Big\|\hat{f}\Big(\frac{\xi}{t-\delta}\Big)\Big\|_{L^2(|\xi|\leq t-\delta)} \\
&\quad +\sum_{2^j\geq t-\delta}2^{\frac{j}{2}}\Big\|\hat{f}\Big(\frac{\xi}{t-\delta}\Big)\Big\|_{L^2(2^j\leq|\xi|\leq2^{j+1})}\Big) \\
&\leq C\Big(\frac{\delta}{t-\delta}\Big)^{\frac{1}{2}}(t-\delta)^{\frac{1}{2}}\Big(\parallel\hat{f}(\xi)\parallel_{L^2(|\xi|\leq1)} \quad +\sum_{j=0}^{\infty}2^{\frac{j}{2}}\parallel\hat{f}(\xi)\parallel_{L^2(2^j\leq|\xi|\leq2^{j+1})}\Big) \\
&\leq C\delta^{\frac{1}{2}}\Big(\parallel\hat{f}(\xi)\parallel_{L^2(|\xi|\leq1)}
\quad +\sum_{j=0}^{\infty}2^{\frac{j}{2}}\parallel\hat{f}(\xi)\parallel_{L^2(2^j\leq|\xi|\leq2^{j+1})}\Big).
\end{split}
\end{equation*}
Thus, the proof of \eqref{equ:a.6} is completed.

\end{proof}

\end{document}